\documentclass[a4paper]{article}

\usepackage[latin1]{inputenc} 
\usepackage{graphicx}
\usepackage[pdftex,bookmarks]{hyperref}
\usepackage{amsmath}    
\usepackage{amssymb}
\usepackage{amsmath,amsthm}
\usepackage{textcomp}
\usepackage{color}
\usepackage[all]{xy}
\usepackage{geometry}
\usepackage{tikz-cd}

\newtheorem{teo}{Theorem}[section]
\newtheorem{defi}{Definition}[section]
\newtheorem{lemma}{Lemma}[section]

\newtheorem{cor}{Corollary}[section]
\newtheorem{prop}{Proposition}[section]

\newtheorem{rem} {Remark}[section]

\newcommand{\vep}{\delta}

\DeclareMathOperator{\dvol}{dvol}

\DeclareMathOperator{\supp}{supp}

\DeclareMathOperator{\vol}{vol}

\title{\huge \bf $L^2$-harmonic forms and spinors on stable minimal hypersurfaces}
\author{Francesco Bei and Giuseppe Pipoli} 
\date{}

\begin{document}

\maketitle

\begin{abstract}
Let $f:N\rightarrow (M,g)$ be a two-sided, complete, stable, minimal, immersed hypersurface. In this paper we establish various vanishing theorems for the space of $L^2$-harmonic forms and spinors (when $M$ is additionally spin)  under suitable positive curvature assumptions on the ambient manifold. Our results in the setting of forms extend to higher dimensions and more general ambient Riemannian manifolds previous vanishing theorems due to Tanno \cite{Tanno} and Zhu \cite{Zhu}. In the setting of spin manifolds our results allow to conclude,  for instance, that any oriented, complete, stable, minimal, immersed hypersurface of $\mathbb{R}^m$ or $\mathbb{S}^m$ carries no non-trivial $L^2$-harmonic spinors. Finally, analogous results are proved for strongly stable constant mean curvature hypersurfaces.
\end{abstract}



\section*{Introduction}
Minimal hypersurfaces are a central topic in both Riemannian geometry and geometric analysis. We recall that given an $m$-dimensional Riemannian manifold $(M,g)$, a minimal hypersurface in $(M,g)$ is an immersed $(m-1)$-dimensional manifold $f:N\rightarrow (M,g)$ that is a critical point of the area functional. Equivalently $f:N\rightarrow (M,g)$ is a minimal immersion if and only if the mean curvature vector field vanishes identically. Furthermore a minimal hypersurface is said to be stable if the second variation of the area functional is non-negative. A fundamental question in this research area is the so-called stable Bernstein problem which asks whether a stable minimal hypersurface
in $\mathbb{R}^{n}$ must be necessarily a hyperplane. This  problem was shown to be true by do Carmo and Peng \cite{do Carmo}, Fischer-Colbrie and Schoen \cite{Fischer} and Pogorelov \cite{Pogorelov} for $n=3$ and recently by Chodosh and Li  \cite{CL1} and Catino, Matrolia and Roncoroni \cite{CMR} for $n=4$, Chodosh, Li, Minter and Stryker \cite{CL3} for $n=5$ and Mazet \cite{Ma} for $n=6$.  It has a negative answer when $n\geq 8$ as a consequence of the work of Bombieri-De Giorgi-Giusti \cite{BDGG} and it is still open if $n=7$. Concerning the stability of an oriented minimal hypersurface $f:N\rightarrow \mathbb{R}^n$, a topological obstruction was found by Palmer \cite{Palmer}, namely the existence of a codimension-one cycle in $N$ that does not separate $N$. According to Dodziuk \cite{Dodziuk}, the presence of such a cycle implies the existence of a non-trivial $L^2$-harmonic $1$-form on $N$. Thus the key point in Palmer's proof is to show that a stable minimal hypersurface $f:N\rightarrow \mathbb{R}^{n}$ carries no non-trivial $L^2$-harmonic $1$-form. These results led Tanno a few years later to formulate the following problem \cite{Tanno}:\\

 Let $f:N\rightarrow \mathbb{R}^{n}$ be  a complete, oriented and stable minimal hypersurface. Are there non-trivial $L^2$-harmonic $p$-forms on $N$?\\
 
In his paper Tanno proved the vanishing of any $L^2$-harmonic $p$-form when $n=5$ and subsequently Zhu \cite{Zhu1} proved an analogue vanishing when the ambient space is $\mathbb{S}^5$ or more generally a complete five-dimensional Riemannian manifold $(M,g)$ with a positive pinching condition on its sectional curvatures \cite{Zhu}. Note that the first vanishing result can now be seen as an immediate consequence of \cite{CL1}, \cite{CL3} and \cite{CMR}, while the other two are empty by \cite{CMR}. To the best of our  knowledge nothing is known in higher dimension. Besides the above problem there is another vanishing-type question that is interesting to investigate. When the ambient space is $\mathbb{R}^n$, $\mathbb{S}^n$ or any other spin manifold $M$ it is well known that every oriented immersed hypersurface $f:N\rightarrow M$ admits spin structures.  Therefore, besides the existence of $L^2$-harmonic $p$-forms, it is also natural to understand whether $N$ supports any non-trivial $L^2$-harmonic spinor. We call this  second problem the ``{\em Spinorial Tanno's problem}''.  The goal of this paper is to tackle these two vanishing problems described above.  Let us give more details by describing how the paper is organized. The first section contains some general vanishing and finiteness results for the $L^2$-kernel of a  Schr\"odinger or Dirac operator over a complete Riemannian manifold in the presence of a weighted Poincar\'e inequality, see Th. \ref{finiteness} and Th. \ref{c-length}. Besides their crucial role played in the geometric applications developed in the second section, we believe that these results have independent interest on their own. In the second section we collect various geometric applications to stable minimal hypersurfaces that follow from the results proved in the first one. The first application is concerned with the {\em nullity} of a two-sided and complete stable, minimal, immersed hypersurface $f:N\rightarrow (M,g)$. We show that if $(M,g)$ has non-negative Ricci curvature then the nullity of $f:N\rightarrow (M,g)$ is necessarily either $0$ or $1$, see Cor. \ref{nullity}. We also investigate the properties of the conformal Laplacian of a two-sided and complete stable minimal immersed hypersurface. We show that if the ambient manifold has non-negative scalar curvature then the conformal Laplacian is essentially self-adjoint with finite dimensional $L^2$-kernel. Moreover we apply these results to show the existence of (possibly incomplete) Riemannian metrics that are conformally equivalent to the metric induced by the immersion and have vanishing scalar curvature or positive scalar curvature, see Props. \ref{conformal}-\ref{scalarflat}. Then we continue by establishing  various vanishing theorems for $L^2$-harmonic spinors and forms over a stable minimal hypersurface. In particular, concerning the vanishing of $L^2$-harmonic spinors, we have the following result, see Th. \ref{spin} and Cor. \ref{imm2}:

\begin{teo}
\label{spinx}
Let $(M,g)$ be a Riemannian manifold with $s_g\geq 0$. Let $N$ be a spinnable manifold with $\dim(N)+1=\dim(M)$ such that there exists a two-sided, stable minimal immersion $$f:N\rightarrow (M,g)$$ with $(N,g_N)$  complete. Let $P_{\mathrm{Spin(n)}}(N)\rightarrow N$ be an arbitrarily fixed Riemannian spin structure on $N$ and let $(\Sigma N,\tau)\rightarrow N$ and $\eth$ be the corresponding spinor bundle and spin-Dirac operator. Then every $L^2$-harmonic spinor of $(\Sigma N, \tau)\rightarrow N$ has constant length and consequently 
$$\dim\left(\ker(\eth)\cap L^2(N,\Sigma N,g_N,\tau)\right)\leq 2^{n/2}$$ if $n$ is even whereas
$$\dim\left(\ker(\eth)\cap L^2(N,\Sigma N,g_N,\tau)\right)\leq 2^{(n-1)/2}$$ if $n$ is odd. If in addition $\vol_{g_N}(N)=+\infty$ then $(N,g_N)$ carries no non-trivial $L^2$-harmonic spinors.
\end{teo}

In particular the above result settles completely the spinorial Tanno's problem in both $\mathbb{R}^m$ and $\mathbb{S}^m$ since it implies that an oriented and complete  stable, minimal, immersed hypersurface of $\mathbb{R}^m$ or $\mathbb{S}^m$ carries no non-trivial $L^2$-harmonic spinors, see Cor. \ref{imm2} and Th. \ref{spin-cor}. Concerning the vanishing of $L^2$-harmonic forms the situation is more involved. We found a condition on the principal curvatures, see \eqref{pcc}, that coupled with some curvature conditions on the ambient manifold, yields various vanishing results. More precisely, see Th. \ref{noL2forms} in the text, we have 
\begin{teo}
\label{noL2formsx}
Let $(M,g)$ be a Riemannian manifold of dimension $m+1$  with sectional curvature $\mathrm{sec}_g$ and curvature operator $\mathcal R_g$ and let $\Sigma$ be an oriented manifold of dimension $m$ with a two-sided, stable minimal immersion $f:\Sigma\rightarrow (M,g)$ such that $(\Sigma,g_{\Sigma})$ is complete. { Let $2\leq p\leq \frac m2$}.
\begin{enumerate}
\item  If $m\geq 4$ and  $\mathcal{R}_g\geq0$ (or more generally there exists $\gamma\in \mathbb{R}$ such that $\mathcal{R}_g\geq \gamma$ and $p(m-p)\gamma+\mathrm{Ric}_g(\mathcal{N},\mathcal{N})\geq 0$) and  $|A|^2-K^2_{\alpha}\geq 0$ for any $\alpha\subset \{1,...,m\}$ with $|\alpha|=p$, then every $L^2$-harmonic $p$-form on $(\Sigma,g_\Sigma)$ has constant length.  If in addition $\Sigma$ is not totally geodesic or $\mathrm{Ric}_g(\mathcal{N},\mathcal{N})$ is somewhere positive on $f(\Sigma)$, then $$\mathcal{H}^p_2(\Sigma,g_{\Sigma})=\mathcal{H}^{m-p}_2(\Sigma,g_{\Sigma})=\{0\}.$$
\item If $m\geq 6$ and $\mathrm{sec}_g\in [a,b]$ with $0<a\leq b\leq \varepsilon_{m,p}a$ with $\varepsilon_{m,p}$ the constant defined in \eqref{nonsharp} and $|A|^2-K^2_{\alpha}\geq 0$ for any $\alpha\subset \{1,...,m\}$ with $|\alpha|=p$, then $$\mathcal{H}^p_2(\Sigma,g_{\Sigma})=\mathcal{H}^{m-p}_2(\Sigma,g_{\Sigma})=\{0\}.$$
\item If $m\geq 6$ and $\mathrm{sec}_g\in [a,b]$ with $0<a\leq b\leq c_ma$, $c_m$ the constant defined in \eqref{sharp} and $|A|^2-K^2_{\alpha}\geq 0$ for any $\alpha\subset \{1,...,m\}$ with $|\alpha|=2$, then $$\mathcal{H}^2_2(\Sigma,g_{\Sigma})=\mathcal{H}^{m-2}_2(\Sigma,g_{\Sigma})=\{0\}.$$
\end{enumerate}
\end{teo}
We point out that under the hypothesis of points $2.$ or $3.$, if $m\leq 5$ there are no complete oriented stable minimal  hypersurfaces, see \cite[Corollary 1.3]{CMR}. 
The remaining part of the second section and the third section contain a thorough discussion about the condition $|A|^2-K^2_{\alpha}\geq 0$ and the curvature condition of the ambient manifold as well as various applications of Th. \ref{noL2formsx}. For the sake of brevity we recall here only that Th. \ref{noL2formsx} implies the following vanishing results:
\begin{itemize}
\item Any oriented and complete, stable, minimal, immersed hypersurface in $\mathbb{S}^3\times \mathbb{R}^2$ or in $\mathbb{S}^2\times \mathbb{R}^3$ carries no non-trivial $L^2$-harmonic forms. 
\item The minimal entire graphs in $\mathbb{R}^{2n+1}$ with $n\geq 4$, constructed by Bombieri-De Giorgi-Giusti \cite{BDGG}, carry no non-trivial $L^2$-harmonic $p$-forms for any $p\leq \sqrt{2n}$.
\item Any oriented and complete, stable, minimal, immersed hypersurface of $\mathbb{S}^m$ carries no non-trivial $L^2$-harmonic $p$-forms  if the second fundamental form is bounded above by a certain constant, see Th. \ref{Abound} and the subsequent discussion.
\end{itemize}

Finally, in the fourth section we briefly discuss how to extend our results to the case of  strongly stable  constant mean curvature hypersurfaces.

\section{Finiteness and vanishing theorems for Schr\"odinger and Dirac operators}

This section is devoted to various vanishing and finiteness results for the $L^2$-kernel of Dirac and Schr\"oedinger operators on certain complete Riemannian manifolds. As we will see later in this paper, these results will play a crucial role in our geometric applications. We start now by introducing some notation and preliminary notions. Let $(N,g)$ be a Riemannian manifold endowed with a Riemannian (or Hermitian) vector bundle $(E,\rho)\rightarrow N$. Let $$P:=\nabla^t\circ \nabla+L$$ be a formally self-adjoint Schr\"odinger operator acting on $C^{\infty}_c(N,E)$, where  $\nabla$ is an arbitrarily fixed metric connection on $(E,\rho)$, $\nabla^t$ is the formal adjoint of $\nabla$ w.r.t. $g$ and $\rho$ and $L\in C^{\infty}(N,\text{End}(E))$ is a fiberwise self-adjoint endomorphism w.r.t. $\rho$. Let $\Delta:C^{\infty}(N,\mathbb{R})\rightarrow C^{\infty}(N,\mathbb{R})$ be the Laplace-Beltrami operator induced by $g$. In this paper we use the definition $\Delta:=d^t\circ d$ with $d^t:\Omega_c^1(N)\rightarrow C_c^{\infty}(N)$ the formal adjoint of $d:C^{\infty}_c(N)\rightarrow \Omega^1_c(N)$ w.r.t. $g$. In particular $\Delta$ is non-negative on $C_c^{\infty}(N,g)$. Finally let $q\in C^{\infty}(N,\mathbb{R})$ be arbitrarily fixed. We have the following 
\begin{teo}
\label{finiteness}
In the above setting assume that  $(N,g)$ is complete and that 
\begin{equation}
\label{1sign}
L+q\geq 0\quad \mathrm{and}\quad\quad \langle(\Delta-q)\phi,\phi\rangle_{L^2(N,g)}\geq 0
\end{equation}
for every $\phi\in C^{\infty}_c(N,\mathbb{R})$. Then $$P:L^2(N,E,g,\rho)\rightarrow L^2(N,E,g,\rho)$$ defined initially on $C^{\infty}_c(N,E)$ is essentially self-adjoint. Moreover if $q\geq 0$ on $N$ then every section $0\neq s\in \ker(P)\cap L^2(N,E,g,\rho)$ is nowhere vanishing on $N$. Consequently $\ker(P)\cap L^2(N,E,g,\rho)$ is finite dimensional and satisfies 
\begin{equation}
\label{xxx}
\dim(\ker(P)\cap L^2(N,E,g,\rho))\leq \mathrm{rnk}(E).
\end{equation}
\end{teo}

Note that in the above inequality \eqref{xxx}, $\dim(...)$ and $\mathrm{rnk}(...)$ denote the real or complex dimension and rank according to the fact that  $E$ is a real or a complex vector bundle. In order to prove the above theorem we need some preliminary results. First of all we recall that $W^{1,2}(N,g)$ is the Sobolev space made by functions $f\in L^2(N,g)$ such that $df$, understood in the distributional sense, lies in $L^2\Omega^1(N,g)$, the Hilbert space of $L^2$ integrable $1$-forms on $(N,g)$. Since $(N,g)$ is complete it is well known that $C^{\infty}_c(N)$ is dense in $W^{1,2}(N,g)$ w.r.t the graph norm of $d$. We denote with $\overline{d}:W^{1,2}(N,g)\rightarrow L^2\Omega^1(N,g)$ the extension of $d:C^{\infty}_c(N)\rightarrow L^2(N,g)$ and with $W^{1,2}_{\mathrm{comp}}(N,g)$ the subspace of $W^{1,2}(N,g)$ made by functions whose distributional support is compact.

\begin{lemma}
\label{W1}
In the setting of Th. \ref{finiteness} we have $$\int_N\left(|\overline{d}\phi|^2_g-q\phi^2\right)\dvol_g\geq 0$$ for any $\phi\in W^{1,2}_{\mathrm{comp}}(N,g)$. If in addition $q\in L^{\infty}(N)$ then the above inequality holds true for any $\phi\in W^{1,2}(N,g)$.
\end{lemma}
\begin{proof}
Let $\phi\in W^{1,2}_{\mathrm{comp}}(N,g)$ and let $U$ be an open neighbourhood of $\mathrm{supp}(\phi)$, the support of $\phi$. Then there exists a sequence $\{\phi_j\}_{j\in \mathbb{N}}\subset C_c^{\infty}(N)$ such that $\phi_j\rightarrow \phi$ in $L^2(N,g)$ as $j\rightarrow +\infty$, $d\phi_j\rightarrow \overline{d}\phi$ in $L^2\Omega^1(N,g)$ as $j\rightarrow +\infty$ and $\mathrm{supp}(\phi_j)\subset U$ for each $j\in \mathbb{N}$. In this way we have $$\int_N|\overline{d}\phi|_g^2\dvol_g=\|\overline{d}\phi\|^2_{L^2\Omega^1(N,g)}=\lim_{j\rightarrow +\infty}\|d\phi_j\|^2_{L^2\Omega^1(N,g)}=\lim_{j\rightarrow +\infty}\int_N|d\phi_j|_g^2\dvol_g$$ and 
\begin{equation}
\label{stima}
\left|\int_Nq(\phi^2-\phi_j^2)\dvol_g\right|\leq \int_U\left|q(\phi^2-\phi_j^2)\right|\dvol_g\leq \|q\|_{L^{\infty}(U)}\|\phi-\phi_j\|_{L^2(N,g)}\|\phi+\phi_j\|_{L^2(N,g)}.
\end{equation}
The above inequality implies that $$\lim_{j\rightarrow +\infty}\int_Nq(\phi^2-\phi_j^2)\dvol_g=0$$ since $$\lim_{j\rightarrow +\infty}\|\phi-\phi_j\|_{L^2(N,g)}=0\quad \text{and}\quad \lim_{j\rightarrow +\infty}\|\phi+\phi_j\|_{L^2(N,g)}=2\|\phi\|_{L^2(N,g)}.$$ We can thus conclude that $$\int_N\left(|\overline{d}\phi|^2_g-q\phi^2\right)\dvol_g\geq 0$$ because $$\int_N\left(|\overline{d}\phi|^2_g-q\phi^2\right)\dvol_g=\lim_{j\rightarrow +\infty}\int_N\left(|d\phi_j|^2_g-q\phi_j^2\right)\dvol_g$$ and $$\int_N\left(|d\phi_j|^2_g-q\phi_j^2\right)\dvol_g=\langle\Delta \phi_j-q\phi_j,\phi_j\rangle_{L^2(N,g)}\geq 0.$$ Let now tackle the case where $q\in L^{\infty}(N)$ and $\phi\in W^{1,2}(N,g)$. Let $\{\phi_j\}_{j\in \mathbb{N}}\subset C_c^{\infty}(N,\mathbb{R})$ be a sequence converging to $\phi$ in $W^{1,2}(N,g)$. The proof follows by repeating the same argument above since we can replace \eqref{stima} with  $$\left|\int_Nq(\phi^2-\phi_j^2)\dvol_g\right|\leq \int_N\left|q(\phi^2-\phi_j^2)\right|\dvol_g\leq \|q\|_{L^{\infty}(N)}\|\phi-\phi_j\|_{L^2(N,g)}\|\phi+\phi_j\|_{L^2(N,g)}.
$$
\end{proof}

\begin{lemma}
\label{solution}
In the setting of Th. \ref{finiteness} let $\psi\in L^2(N,g)$ and let $\{\psi_j\}_{j\in \mathbb{N}}\subset W^{1,2}_{\mathrm{comp}}(N,g)$ be a sequence such that 
$$
\lim_{j\rightarrow +\infty}\|\psi-\psi_j\|_{L^2(N,g)}=0\quad \text{and}\quad \lim_{j\rightarrow +\infty}\int_N \left(|\overline{d}\psi_j|^2_{g}-q|\psi_j|^2\right)\dvol_g=0.$$
 Then  $$\psi\in \ker(H)$$ with $H:L^2(N,g)\rightarrow L^2(N,g)$ the unique $L^2$-closed extension of $\Delta-q:C^{\infty}_c(N)\rightarrow L^2(N,g)$. Consequently $\psi\in C^{\infty}(N)\cap \ker(\Delta-q)$. 
\end{lemma}

\begin{proof}
Let $Q_{\Delta-q}$ be the symmetric bilinear form associated to $\Delta-q$, that is 
$Q_{\Delta-q}:C^{\infty}_c(N)\times C^{\infty}_c(N)\rightarrow \mathbb{R}$ and $$Q(\chi,\varphi)=\langle\Delta\chi-q\chi,\varphi\rangle_{L^2(N,g)}=\int_N\left(g(d\chi,d\varphi)-q\chi\varphi\right)\dvol_g.$$
Let $\overline{Q}_{\Delta-q}$ be the closure of $Q_{\Delta-q}$. We recall that $\overline{Q}_{\Delta-q}$ can be described as follows. Let $H:L^2(N,g)\rightarrow L^2(N,g)$ be the closed operator defined as the $L^2$-graph closure of $\Delta-q:C^{\infty}_c(N)\rightarrow L^2(N,g)$.
Since  $(N,g)$ is complete and $\Delta-q$ is  non-negative and formally self-adjoint on $C^{\infty}_c(N)$, we know that $\Delta-q:L^2(N,g)\rightarrow L^2(N,g)$, with initial domain $C^{\infty}_c(N)$, is essentially self-adjoint in $L^2(N,g)$, see e.g. \cite[Th. 3.13]{Pigola}. Therefore $H$ is the unique $L^2$-closed extension of $\Delta-q:C^{\infty}_c(N)\rightarrow L^2(N,g)$. In particular $H$ is a non-negative self-adjoint operator. Let now $H^{\frac{1}{2}}$ be the square root of $H$. We recall that $H^{\frac{1}{2}}:L^2(N,g)\rightarrow L^2(N,g)$ is the unique self-adjoint and non-negative operator such that $(H^{\frac{1}{2}})^2=H$. In particular $\mathrm{Dom}(H)$, the domain of $H:L^2(N,g)\rightarrow L^2(N,g)$, can be described as $$\mathrm{Dom}(H)=\{f\in \mathrm{Dom}(H^{\frac{1}{2}}): H^{\frac{1}{2}}f\in \mathrm{Dom}(H^{\frac{1}{2}})\}\quad \text{and}\quad Hf=H^{\frac{1}{2}}(H^{\frac{1}{2}}(f)).$$ Note that in particular the above characterization implies 
\begin{equation}
\label{ker0.5}
\ker(H)=\ker(H^{\frac{1}{2}}).
\end{equation}
Then, going back to $\overline{Q}_{\Delta-q}$, we have $$\mathrm{Dom}(\overline{Q}_{\Delta-q})=\mathrm{Dom}(H^{\frac{1}{2}})\quad \text{and}\quad \overline{Q}_{\Delta-q}(u,v)=\langle H^{\frac{1}{2}}u,H^{\frac{1}{2}}v\rangle_{L^2(N,g)}$$ for every $u,v\in \mathrm{Dom}(H^{\frac{1}{2}})$, see e.g. \cite[p. 193]{Gun}. Let now $\gamma\in W^{1,2}_{\mathrm{comp}}(N,g)$. We want to show that $\gamma\in \mathrm{Dom}(\overline{Q}_{\Delta-q})$. To this aim it is enough to show the existence of a sequence $\{\gamma_j\}_{j\in \mathbb{N}}\subset C^{\infty}_c(N)$ such that $\gamma_j\rightarrow \gamma$ in $L^2(N,g)$, as $j\rightarrow \infty$ and $Q_{\Delta-q}(\gamma_j-\gamma_i,\gamma_j-\gamma_i)\rightarrow 0$ as $i,j\rightarrow +\infty$, see e.g. \cite[Prop. 10.1]{Konrad}. Let $U$ be an open neighbourhood of  $\mathrm{\supp}(\gamma)$ and let $\{\gamma_j\}_{j\in \mathbb{N}}\subset C_c^{\infty}(U)$ be a sequence such that $\gamma_j\rightarrow \gamma$ in $L^2(N,g)$ as $j\rightarrow +\infty$, $d\gamma_j\rightarrow \overline{d}\gamma$ in $L^2\Omega^1(N,g)$ as $j\rightarrow +\infty$. Then, arguing as in the proof of Lemma \ref{W1}, we have  $$\lim_{i,j\rightarrow +\infty}\int_N\left(|d\gamma_i-d\gamma_j|_g^2-q|\gamma_i-\gamma_j|^2\right)\dvol_g=0$$ and so we can conclude that $\gamma\in \mathrm{Dom}(\overline{Q}_{\Delta-q})$ and 
$$\overline{Q}_{\Delta-q}(\gamma,\gamma)=\int_N\left(|\overline{d}\gamma|^2_{\gamma}-q\gamma^2\right)\dvol_g$$ for each $\gamma\in W^{1,2}_{\text{com}}(N,g)$.
Now let's go back to the sequence $\{\psi_j\}_{j\in \mathbb{N}}$. We know that $\psi_j\in \mathrm{Dom}(\overline{Q}_{\Delta-q})$, $\psi_j\rightarrow \psi$ in $L^2(N,g)$ as $j\rightarrow +\infty$ and $\overline{Q}_{\Delta-q}(\psi_j,\psi_j)\rightarrow 0$ as $j\rightarrow +\infty$. We can rephrase these three properties by saying that $\psi_j\in \mathrm{Dom}(H^{\frac{1}{2}})$, $\psi_j\rightarrow \psi$ in $L^2(N,g)$ as $j\rightarrow +\infty$ and $H^{\frac{1}{2}}\psi_j\rightarrow 0$ in $L^2(N,g)$ as $j\rightarrow +\infty$. Since $H^{\frac{1}{2}}$ is a closed operator this amounts to saying that $\psi\in \ker(H^{\frac{1}{2}})$ and finally, thanks to \eqref{ker0.5}, we can conclude that $\psi\in \ker(H)$, as required. To conclude the proof we note that since $\Delta-q$ is elliptic, we have $\psi\in C^{\infty}(N)$ and thus $H\psi=\Delta\psi-q\psi=0$.
\end{proof}

We can now prove Th. \ref{finiteness}.

\begin{proof}
As a first step to show that $P:C_c^{\infty}(N,E)\rightarrow L^2(N,E,g,\rho)$ is essentially self-adjoint we want to prove that there exists a constant $c\in \mathbb{R}$ such that $$\langle Ps,s\rangle_{L^2(N,E,g,\rho)}\geq c\|s\|^2_{L^2(N,E,g,\rho)}$$ for each $s\in C_c^{\infty}(N,E)$. So let $s\in C_c^{\infty}(N,E)$. We have 
$$
\begin{aligned}
\langle Ps,s\rangle_{L^2(N,E,g,\rho)}&=\int_N\rho(Ps,s)\dvol_g\\
&=\int_N \left(|\nabla s|^2_{g\otimes \rho}+\rho(Ls,s)\right) \dvol_g\\
(\text{by Kato's inequality}) &\geq \int_N \left(|\overline{d}|s|_{\rho}|^2_{g}+\rho(Ls,s)\right) \dvol_g\\
&=\int_N \left(|\overline{d}|s|_{\rho}|^2_{g}-q|s|^2_{\rho}+q|s|^2_{\rho}+\rho(Ls,s)\right) \dvol_g\\
&=\int_N \left(|\overline{d}|s|_{\rho}|^2_{g}-q|s|^2_{\rho}+\rho((L+q)s,s)\right) \dvol_g\\
&\geq0.
\end{aligned}
$$
Note that the last inequality follows by Lemma \ref{W1} and the fact that $L+q\geq 0$. Now let us consider the operator $\hat{P}:=P+\mathrm{Id}$, with $\mathrm{Id}:C^{\infty}(N,E)\rightarrow C^{\infty}(N,E)$ the identity. It is clear that $P$ is essentially self-adjoint if and only $\hat{P}$ is so. Since $\hat{P}$ is elliptic, symmetric and satisfies the inequality $$\langle \hat{P}s,s\rangle_{L^2(N,E,g,\rho)}\geq \|s\|^2_{L^2(N,E,g,\rho)}$$ for each $s\in C_c^{\infty}(N,E)$, in order to conclude that $\hat{P}:C_c^{\infty}(N,E)\rightarrow L^2(N,E,g,\rho)$ is essentially self-adjoint in $L^2(N,E,g,\rho)$, it is enough to show that $\ker(\hat{P})\cap L^2(N,E,g,\rho)=\{0\}$, see e.g. \cite[p. 136-137]{RS2}. To show that $\ker(\hat{P})\cap L^2(N,E,g,\rho)$ is trivial we adapt to our context an argument that is well known in the scalar case, see e.g. \cite[Th. 3.13]{Pigola}. Let $s\in \ker(\hat{P})\cap L^2(N,E,g,\rho)$, let $\{V_n\}_{n\in \mathbb{N}}$ be an exhaustion of $N$ given by relatively compact open subset $V_n$ and let $\{\phi_n\}_{n\in \mathbb{N}}\subset C^{\infty}_c(N)$ be a sequence of smooth functions with compact support, such that 
\begin{enumerate}
\item $0\leq \phi_n\leq 1$
\item $\phi_n|_{V_n}\equiv 1$
\item $\|d\phi_n\|_{L^{\infty}\Omega^1(N,g)}\rightarrow 0$ as $n\rightarrow +\infty$.
\end{enumerate}
Since $(N,g)$ is complete it is well known that such a sequence exists. We then have 
$$
\begin{aligned}
0=\langle \hat{P}s,\phi_n^2s\rangle_{L^2(N,E,g,\rho)}&=\int_N \left(\langle\nabla s,\nabla(\phi_n^2s)\rangle_{g\otimes \rho}+\rho(Ls+s,\phi^2_ns)\right) \dvol_g\\
&= \int_N \left(\langle\nabla s,2\phi_nd\phi_n\otimes  s+\phi_n^2\nabla s \rangle_{g\otimes \rho}+\rho(Ls+s,\phi_n^2s)\right) \dvol_g\\
&=\int_N \left(\langle\phi_n\nabla s +d\phi_n\otimes s,d\phi_n\otimes s+\phi_n \nabla s \rangle_{g\otimes \rho}-|d\phi_n\otimes s|^2_{g\otimes \rho}+\rho(Ls+s,\phi_n^2s)\right) \dvol_g\\
&=\int_N \left(\langle\nabla(\phi_n s) , \nabla(\phi_n s) \rangle_{g\otimes \rho}-|d\phi_n\otimes s|^2_{g\otimes \rho}+\rho(Ls+s,\phi_n^2s)\right) \dvol_g\\
&=\int_N \left(\hat{P}(\phi_n s) , \phi_n s \rangle_{g\otimes \rho}-|d\phi_n|^2_g|s|^2_{\rho}\right) \dvol_g\\
&\geq \int_N \phi_n^2 |s|^2_h\dvol_g-\int_N |d\phi_n|^2_g|s|^2_{\rho} \dvol_g.
\end{aligned}
$$
Letting $n\rightarrow +\infty$ we obtain $s=0$. The first statement of theorem is thus proved. Let us now consider $\ker(P)\cap L^2(N,E,g,\rho)$. Let $s\in \ker(P)\cap L^2(N,E,g,\rho)$. Since $P$ is essentially self-adjoint there exists $\{s_j\}_{j\in \mathbb{N}}\subset C_c^{\infty}(N,E)$ such that $s_j\rightarrow s$  and $Ps_j \rightarrow 0$ both in $L^2(N,E,g,\rho)$  as $j\rightarrow \infty$. We then have 
$$
\begin{aligned}
0=\langle Ps,s\rangle_{L^2(M,E,g,\rho)}&=\lim_{j\rightarrow +\infty}\langle Ps_j,s_j\rangle_{L^2(M,E,g,\rho)}\\
&=\lim_{j\rightarrow +\infty}\int_N \left(|\nabla s_j|^2_{g\otimes \rho}+\rho(Ls_j,s_j)\right) \dvol_g\\
 &\geq \lim_{j\rightarrow +\infty}\int_N \left(|\overline{d}|s_j|_{\rho}|^2_{g}+\rho(Ls_j,s_j)\right) \dvol_g\\
&=\lim_{j\rightarrow +\infty}\int_N \left(|\overline{d}|s_j|_{\rho}|^2_{g}-q|s_j|^2_{\rho}+q|s_j|^2_{\rho}+\rho(Ls_s,s_j)\right) \dvol_g\\
&=\lim_{j\rightarrow +\infty}\int_N \left(|\overline{d}|s_j|_{\rho}|^2_{g}-q|s_j|^2_{\rho}+\rho((L+q)s_j,s_j)\right) \dvol_g.
\end{aligned}
$$
In this way we get 
$$
\lim_{j\rightarrow +\infty}\int_N \left(|\overline{d}|s_j|_{\rho}|^2_{g}-q|s_j|^2_{\rho}\right)\dvol_g=0\quad \text{and}\quad \lim_{j\rightarrow +\infty}\int_N \left(\rho((L+q)s_j,s_j)\right) \dvol_g=0.
$$
In particular for the sequence $\{|s_j|_{\rho}\}_{j\in \mathbb{N}}$ we have 
\begin{equation}
\label{null}
|s_j|_{\rho}\in W^{1,2}_{\text{comp}}(N,g),\quad \lim_{j\rightarrow +\infty}\||s|_{\rho}-|s_j|_{\rho}\|_{L^2(N,g)}=0\quad \text{and}\quad \lim_{j\rightarrow +\infty}\int_N \left(|\overline{d}|s_j|_{\rho}|^2_{g}-q|s_j|^2_{\rho}\right)\dvol_g=0.
\end{equation}
According to Lemma \ref{solution} we know that \eqref{null} implies that $|s|_{\rho}\in \ker(\Delta-q)$ and since $\Delta-q$ is an elliptic operator we deduce that $|s|_{\rho}$ is smooth on $N$. Assume now that $q\geq 0$ on $N$. Then $\Delta |s|_{\rho}=q|s|_{\rho}\geq 0$ and so by the maximum principle applied to $\Delta$ and $|s|_g$, see e.g. \cite[Th. A.2]{Lee}, we get that either  $s\equiv 0$ on $N$ or $s$ is nowhere vanishing on $N$. Let us conclude now by showing that $\ker(P)\cap L^2(N,E)$ is finite dimensional with $\dim(\ker(P)\cap L^2(N,E))\leq \mathrm{rnk}(E)$. Let $s_1,...,s_q\in \ker(P)\cap L^2(N,E)$ with $q>\mathrm{rnk}(E)$ and let $x\in N$ be an arbitrarily fixed point. Then there exists $\lambda_1,...,\lambda_q\in \mathbb{R}$  (or $\mathbb{C}$ if $E$ is a complex vector bundle) with $(\lambda_1,...,\lambda_{q})\neq (0,...,0)$ such that $\lambda_1s_1(x)+...+\lambda_qs_q(x)=0$. Since $\lambda_1s_1+...+\lambda_qs_q\in \ker(P)\cap L^2(N,E)$ we can conclude that  $\lambda_1s_1+...+\lambda_qs_q$ is identically zero on $N$ as $\lambda_1s_1+...+\lambda_qs_q\in \ker(P)$  and vanishes on $x$. 
\end{proof}

\begin{cor}
In the setting of Th. \ref{finiteness} assume additionally that  $q\geq 0$.  If $\dim(\ker(P)\cap L^2(N,E,g,\rho))\geq 1$ then for any non-trivial $s,z\in \ker(P)\cap L^2(N,E,g,\rho))$ there exists $\lambda\in \mathbb{R}$, $\lambda>0$ such that $|s|_{\rho}=\lambda|z|_{\rho}$.
\end{cor}

\begin{proof}
Let $s$, $z$ be non-trivial elements in $\ker(P)\cap L^2(N,E,g,\rho)$. Note that Th. \ref{finiteness} applies in particular to $P=\Delta-q$ with $L=-q$ and thus we have $\dim(\ker(\Delta-q)\cap L^2(N,g))\leq 1$. Therefore, according to the above proof, $|s|_{\rho}$ and $|z|_{\rho}$ are nowhere zero functions lying in $\ker(\Delta-q)\cap L^2(N,g)$ and so there must exists $\lambda\in \mathbb{R}$, $\lambda>0$ such that $|s|_{\rho}=\lambda|z|_{\rho}$, as required.
\end{proof}

\begin{cor}
\label{zerozero} In the setting of Th. \ref{finiteness} assume in addition that $\ker(\Delta-q)\cap L^2(N,g)=\{0\}$. Then $\ker(P)\cap L^2(N,E,g,\rho)=\{0\}$.
\end{cor}

\begin{proof}
According to the proof of Th. \ref{finiteness} we know that if $s\in \ker(P)\cap L^2(N,E,g,\rho)$ then $|s|_{\rho}\in \ker(\Delta-q)\cap L^2(N,g)$. Now the conclusion follows immediately.
\end{proof}

\begin{rem}
\label{yyy}
Clearly the conclusion \eqref{xxx} holds also if we assume $q\leq 0$. Indeed this assumption forces $L$ to be non-negative and thus every section in $\ker(P)\cap L^2(N,E,g,\rho)$ is parallel and lies in $W^{1,2}(M,E,g,\rho)$, see \cite[Th. 1]{Dod}.
\end{rem}
We consider now the case of a Dirac operator. As we will see in a moment, in this setting we have a stronger result. For the sake of brevity we recall only what is strictly necessary for our purpose and we invite the unfamiliar reader to consult excellent monographs on this topic such as \cite{Lawson}.
Let $(N,g)$ be a Riemannian manifold endowed with a Clifford module $(E,\rho)\rightarrow N$. Let $\nabla:C^{\infty}(N,E)\rightarrow C^{\infty}(N,E\otimes T^*N)$ be a Clifford connection and let  $D:C^{\infty}(N,E)\rightarrow C^{\infty}(N,E)$ be the corresponding Dirac operator. We recall that $D$ is a formally self-adjoint first order elliptic differential operator. Moreover its square is a Schr\"oedinger-type operator. More precisely, thanks to the Weitzenb\"ock formula for Dirac operators, see e.g. \cite[Th. 8.2]{Lawson},   we know that there exists $L\in C^{\infty}(N,\mathrm{End}(E))$ such that $$\Delta_D=\nabla^t\circ \nabla +L$$ with $\Delta_{D}=D^2$. As in Th. \ref{finiteness}, let  $\Delta:C^{\infty}(N,g)\rightarrow C^{\infty}(N,g)$ be the Laplace-Beltrami operator on $(N,g)$ and let $q\in C^{\infty}(N,\mathbb{R})$.

\begin{teo}
\label{c-length}
In the above setting assume that  $(N,g)$ is complete and that 
\begin{equation}
\label{sign}
L+q\geq 0\quad \mathrm{and}\quad\quad \langle(\Delta-q)\phi,\phi\rangle_{L^2(N,g)}\geq 0
\end{equation}
for every $\phi\in C^{\infty}_c(N)$. Then every $s\in \ker(D)\cap L^2(N,E,g,\rho)$ has constant length. Consequently $\ker(D)\cap L^2(N,E,g,\rho)$ is finite dimensional and satisfies $$\dim(\ker(D)\cap L^2(N,E,g,\rho))\leq \mathrm{rnk}(E).$$
\end{teo}

\begin{proof}
Let $\{V_n\}_{n\in \mathbb{N}}$ be an exhaustion of $N$ given by relatively compact open subset $V_n$. As in the proof of Th. \ref{finiteness} let us consider a sequence of smooth functions with compact support, $\{\phi_n\}_{n\in \mathbb{N}}\subset C^{\infty}_c(N)$, such that 
\begin{enumerate}
\item $0\leq \phi_n\leq 1$
\item $\phi_n|_{V_n}\equiv 1$
\item $\|d\phi_n\|_{L^{\infty}\Omega^1(N,g)}\rightarrow 0$ as $n\rightarrow +\infty$.
\end{enumerate}
Let now $s\in \ker(D)\cap L^2(N,E,g,\rho)$ and $V_k\in \{V_n\}_{n\in \mathbb{N}}$ be arbitrarily fixed. We have 
$$
\begin{aligned}
0&=\langle Ds,Ds\rangle_{L^2(N,E,g,\rho)}=\lim_{n\rightarrow \infty}\langle D(\phi_ns),D(\phi_ns)\rangle_{L^2(N,E,g,\rho)}\\
&=\lim_{n\rightarrow \infty}\langle\Delta_D(\phi_ns),\phi_ns\rangle_{L^2(N,E,g,\rho)}=\lim_{n\rightarrow \infty}\langle(\nabla^t\circ \nabla+L)(\phi_ns),\phi_ns\rangle_{L^2(N,E,g,\rho)}\\
&=\lim_{n\rightarrow \infty}\int_N\langle(\nabla^t\circ \nabla+L)(\phi_ns),\phi_ns\rangle_{\rho}\dvol_g=\lim_{n\rightarrow \infty}\int_N| \nabla(\phi_ns)|_{\rho\otimes g}^2+\phi_n^2\rho(Ls,s)\dvol_g\\
&= \lim_{n\rightarrow \infty}\left(\int_{N\setminus V_k}| \nabla(\phi_ns)|_{\rho\otimes g}^2+\phi_n^2\rho(Ls,s)\dvol_g+\int_{V_k}| \nabla s|_{\rho}^2+\rho(Ls,s)\dvol_g\right)
\end{aligned}
$$

Note that on $V_k$ we can use the refined Kato inequality \cite{Refined}: there exists a constant $c>0$ (possibly depending also on $V_k$) such that $|\nabla s|_{\rho\otimes g}\geq (1+c)|d|s|_{\rho}|_g$ pointwise over $V_k$. Hence, by using the classical Kato inequality on $N\setminus V_k$ we get
$$
\begin{aligned}
0&\geq \lim_{n\rightarrow \infty}\left(\int_{N\setminus V_k}| \nabla(\phi_ns)|_{\rho\otimes g}^2+\phi_n^2\rho(Ls,s)\dvol_g+\int_{V_k}| \nabla s|_{\rho}^2+\rho(Ls,s)\dvol_g\right)\\
& \geq\lim_{n\rightarrow \infty}\left(\int_{N\setminus V_k}|\overline{d}|\phi_ns|_{\rho}|_g^2+\phi_n^2\rho(Ls,s)\dvol_g+\int_{V_k}(1+c)^2|\overline{d}|s|_{\rho}|_g^2+\rho(Ls,s)\dvol_g\right)\\
& \geq\lim_{n\rightarrow \infty}\left(\int_{N\setminus V_k}|\overline{d}|\phi_ns|_{\rho}|_g^2+\phi_n^2\rho(Ls,s)\dvol_g+\int_{V_k}(1+c)|\overline{d}|s|_{\rho}|_g^2+\rho(Ls,s)\dvol_g\right)\\
& \geq \lim_{n\rightarrow \infty}\left(\int_{N\setminus V_k}|\overline{d}|\phi_ns|_{\rho}|_g^2+\phi_n^2\rho(Ls,s)\dvol_g+\int_{V_k}c|\overline{d}|s|_{\rho}|_g^2\dvol_g+\int_{V_k}|\overline{d}|\phi_ns|_{\rho}|_g^2\dvol_g+\int_{V_k}\phi_n^2\rho(Ls,s)\dvol_g\right)\\
&=\lim_{n\rightarrow \infty}\left(\int_{N}|\overline{d}|\phi_ns|_{\rho}|_g^2+\rho(L(\phi_ns),\phi_ns)\dvol_g+\int_{V_k}c|\overline{d}|s|_{\rho}|_g^2\dvol_g\right)\\
&=\lim_{n\rightarrow \infty}\left(\int_{N}|\overline{d}|\phi_ns|_{\rho}|_g^2+\rho(L(\phi_ns)-q\phi_ns+q\phi_ns,\phi_ns)\dvol_g+\int_{V_k}c|\overline{d}|s|_{\rho}|_g^2\dvol_g\right)\\
&=\lim_{n\rightarrow \infty}\left(\int_{N}|\overline{d}|\phi_ns|_{\rho}|_g^2-q|\phi_ns|^2_{\rho}+\rho(L(\phi_ns)+q\phi_ns,\phi_ns)\dvol_g+\int_{V_k}c|\overline{d}|s|_{\rho}|_g^2\dvol_g\right)
\end{aligned}
$$

By \eqref{sign} we know that 
\begin{equation}
\label{non-neg} \int_{N}(|\overline{d}|\phi_ns|_{\rho}|_g^2-q|\phi_ns|^2_{\rho})\dvol_g\geq 0\quad\quad \mathrm{and}\quad\quad \int_N(\rho(L(\phi_ns)+q\phi_ns,\phi_ns))\dvol_g\geq 0.
\end{equation}

Hence we have 

$$0\geq \lim_{n\rightarrow \infty}\left(\int_{N}|\overline{d}|\phi_ns|_{\rho}|_g^2-q|\phi_ns|^2_{\rho}+\rho(L(\phi_ns)+q\phi_ns,\phi_ns)\dvol_g+\int_{V_k}c|\overline{d}|s|_{\rho}|_g^2\dvol_g\right)\geq 0$$ and thus, using \eqref{non-neg}
we can conclude that 
\begin{equation}
\label{zerozeroo}
\lim_{n\rightarrow \infty}\int_{N}(|\overline{d}|\phi_ns|_{\rho}|_g^2-q|\phi_ns|^2_{\rho})\dvol_g=0,\quad\quad \lim_{n\rightarrow \infty}\int_N(\rho(L(\phi_ns)+q\phi_ns,\phi_ns))\dvol_g=0
\end{equation}
and 
\begin{equation}\label{zerozero}
\int_{V_k}c|\overline{d}|s|_{\rho}|_g^2\dvol_g=0.
\end{equation}
We can thus deduce that $|s|_{\rho}$ is constant on $V_k$ and since $V_k$ is arbitrarily fixed and $|s|_{\rho}$ is continuous, we can conclude that $|s|_{\rho}$ is constant on $N$, as required. Finally arguing as in Th. \ref{finiteness} we get the desired upper bound on $\dim(\ker(D)\cap L^2(N,E,g,\rho))$. 
\end{proof}

We have the following consequences:

\begin{cor}
\label{immediately}
In the setting of Th. \ref{c-length}:
\begin{enumerate}
\item If $F$ is a subbundle of $E$ then $$\dim(\ker(D)\cap L^2(N,F,g,\rho))\leq \mathrm{rnk}(F).$$
\item If $(N,g)$ has infinite volume then $$\ker(D)\cap L^2(N,E,g,\rho)=\{0\}.$$
\item If $L+q$ is positive in (at least) one point then $$\ker(D)\cap L^2(N,E,g,\rho)=\{0\}.$$
\item If $q\geq 0$ and $q(x)>0$ for some $x\in N$ then $\vol_g(N)=+\infty$ and consequently  $$\ker(D)\cap L^2(N,E,g,\rho)=\{0\}.$$
\end{enumerate}
\end{cor}

\begin{proof}
The first two statements follow immediately by Th. \ref{c-length}. Concerning the third one, by \eqref{zerozeroo} we get $$\lim_{n\rightarrow \infty}\int_N\left(\rho(L(\phi_ns)+q\phi_ns,\phi_ns)\right)\dvol_g=0.$$ Since $\rho(L(\phi_ns)+q\phi_ns,\phi_ns)\geq0$ the above limit means that $$\left\|\rho(L(\phi_ns)+q\phi_ns,\phi_ns)\|_{\rho}^2\right\|_{L^1(N,g_N)}\rightarrow 0$$ as $n\rightarrow 0$ and thus, possibly by replacing $\{\phi_n\}_{n\in \mathbb{N}}$ with a subsequence if necessary, we have $$\rho(L(\phi_ns)+q\phi_ns,\phi_ns)\rightarrow 0$$ pointwise a.e. on $N$ as $n\rightarrow \infty$. On the other hand $$\rho(L(\phi_ns)+q\phi_ns,\phi_ns)\rightarrow \rho(Ls+qs,s)$$ pointwise a.e. on $N$ as $n\rightarrow \infty$. Since $L+q\geq 0$ on $N$, it is positive in at least one point and $|s|_{\rho}$ is constant, we can conclude that $s$ vanishes, as required. 
 Finally let us consider the fourth point above. Let us pick a sequence of function $\{\phi_n\}_{n\in \mathbb{N}}$ as in the proof of Th. \ref{c-length}. We have 
$$0\leq \int_N q|\phi_n|^2\dvol_g\leq \int_N |d\phi_n|^2_g\dvol_g\leq \|d\phi_n\|^2_{L^{\infty}\Omega^1(N,g)}\vol_g(\mathrm{supp}(\phi_n)).$$ By the fact that $$\liminf_{n\rightarrow +\infty}\int_N q|\phi_n|^2\dvol_g>0\quad\quad \mathrm{and}\quad\quad \lim_{n\rightarrow +\infty} \|d\phi_n\|_{L^{\infty}\Omega^1(N,g)}=0$$ we deduce that $$\liminf_{n\rightarrow +\infty}\vol_g(\mathrm{supp}(\phi_n))=+\infty\quad\quad \mathrm{and\ thus}\quad\quad \vol_g(N)=+\infty.$$
We can thus conclude that $\ker(D)\cap L^2(N,E,g,\rho)=\{0\}$ since each $s\in \ker(D)\cap L^2(N,E,g,\rho)$ has constant length.
\end{proof}

\section{Applications to stable minimal hypersurfaces}\label{sez2}

In this section we collect various applications of Th. \ref{finiteness} and Th. \ref{c-length} to stable minimal hypersurfaces. Let $(M,g)$ be a Riemannian manifold. We denote with $\mathcal{R}_g$, $R_g$, $\mathrm{sec}_g$, $\mathrm{Ric}_g$ and $s_g$ the curvature operator, curvature tensor, sectional curvature, Ricci tensor and scalar curvature of $(M,g)$, respectively. Moreover, given an immersed submanifold $f:N\rightarrow (M,g)$, we denote with $g_N$, $\mathcal{N}$ and $A$ the metric induced on $N$ by $g$ through $f$, the corresponding unit normal vector field and  second fundamental form, respectively. Given a Riemannian manifold $(M,g)$, an immersed submanifold $f:N\rightarrow (M,g)$ is said to be {\em minimal} if the first derivative of the area functional w.r.t. any smooth normal variation with compact support is zero. Geometrically this means that the mean curvature  of $f:N\rightarrow (M,g)$ vanishes identically. A minimal immersion $f:N\rightarrow (M,g)$ is said to be {\em stable} if in addition the second derivative of the area functional w.r.t. any smooth normal variation with compact support is non-negative. When $f:N\rightarrow M$  has codimension one and is two-sided, the stability condition is equivalent to require that the operator $P:=\Delta-|A|-\mathrm{Ric}_g(\mathcal{N},\mathcal{N})$ satisfies the condition $$\Delta-|A|-\mathrm{Ric}_g(\mathcal{N},\mathcal{N})\geq 0$$ on $C^{\infty}_c(N)$, see e.g. \cite[Eq. (1.147)]{ColMin}. The above operator $P$ is usually named the {\em stability operator} of the immersion $f:N\rightarrow (M,g)$. We come now to the first relevant application of our previous results, which  is concerned with the nullity of a two-sided stable minimal immersed hypersurface. We recall that  the {\em nullity} of a two-sided, immersed, minimal hypersurface $f:N\rightarrow (M,g)$, here denoted by $\mathrm{nullity}(N,f,M,g)$,  is defined as the (possibly infinite) dimension of $\ker(P)\cap L^2(N,g_N)$, with $P$ the corresponding stability operator, see e.g. \cite{CLi}.

\begin{prop}
Let $(M,g)$ be a Riemannian manifold of dimension $m$ and let $f:N\rightarrow (M,g)$ be a complete, two-sided, stable, minimal immersed hypersurface.  If either $|A|^2+\mathrm{Ric}_g(\mathcal{N},\mathcal{N})\leq 0$  or $|A|^2+\mathrm{Ric}_g(\mathcal{N},\mathcal{N})\geq 0$  then  $$ \mathrm{nullity}(N,f,M,g)\in \{0,1\}.$$
\end{prop}

\begin{proof}
Assume first that $|A|^2+\mathrm{Ric}_g(\mathcal{N},\mathcal{N})\geq 0$. Then the conclusion follows by applying  Th. \ref{finiteness} to $P=\Delta-|A|^2-\mathrm{Ric}_g(\mathcal{N},\mathcal{N})$ with  $L=-|A|^2-\mathrm{Ric}_g(\mathcal{N},\mathcal{N})$ and $q=|A|^2+\mathrm{Ric}_g(\mathcal{N},\mathcal{N})$. If $|A|^2+\mathrm{Ric}_g(\mathcal{N},\mathcal{N})\leq 0$  the conclusion  follows  by applying Th. \ref{finiteness} and Rmk \ref{yyy} to $P=\Delta-|A|^2-\mathrm{Ric}_g(\mathcal{N},\mathcal{N})$ with  $L=-|A|^2-\mathrm{Ric}_g(\mathcal{N},\mathcal{N})$ and $q=|A|^2+\mathrm{Ric}_g(\mathcal{N},\mathcal{N})$.
\end{proof}

We have the following immediate consequence:

\begin{cor}
\label{nullity}
Let $(M,g)$ be a Riemannian manifold of dimension $m$ and let $f:N\rightarrow (M,g)$ be a complete, two-sided, stable, minimal immersed hypersurface. If either $\mathrm{Ric}_g\geq 0$ or  $\mathrm{Ric}_g\leq -b^2<0$ and $|A|\leq b^2$ then $$\mathrm{nullity}(N,f,M,g)\in \{0,1\}.$$ 
\end{cor}

We have now some applications to the conformal Laplacian, also known as the Yamabe operator. We point out that in the case $N$ is compact the results of the next two propositions are due to Schoen and Yau, see \cite{SchoenYau}.

\begin{prop}
\label{conformal}
Let $(M,g)$ be a Riemannian manifold of dimension $m$ and let $f:N\rightarrow (M,g)$ be a complete, two-sided, stable, minimal immersed hypersurface. If $s_g\geq 0$ then $$L:=\frac{4(n-1)}{n-2}\Delta+s_{g_N}:L^2(N,g_N)\rightarrow L^2(N,g_N)$$ is essentially self-adjoint with initial domain $C_c^{\infty}(N)$, non-negative and $$\dim(\ker(L)\cap L^2(N,g_N))\leq 1.$$
\end{prop}

\begin{proof}
We adapt to our framework the strategy used by Schoen-Yau in \cite{SchoenYau}. Let us consider the operator $\tilde{L}:=\frac{n-2}{4(n-1)}L$. Obviously the above statement holds true for $L$ if and only if holds true for $\tilde{L}$. Let $q:N\rightarrow \mathbb{R}$ be defined as  $q:=\frac{n-2}{2(n-1)}(\mathrm{Ric_g(\mathcal{N},\mathcal{N})}+|A|^2)$. Since $f:N\rightarrow (M,g)$ is stable and two-sided, we have $\Delta-q\geq 0$ on $C_c^{\infty}(N)$. Indeed given any $\phi\in C_c^{\infty}(N)$ we have
$$
\begin{aligned}
0&\leq \langle d\phi,d\phi\rangle_{L^2\Omega^1(N,g_N)}-\langle (\mathrm{Ric_g(\mathcal{N},\mathcal{N})}+|A|^2)\phi,\phi\rangle_{L^2(N,g_N)}\\
&\leq \frac{2(n-1)}{n-2}\langle d\phi,d\phi\rangle_{L^2\Omega^1(N,g_N)}-\langle (\mathrm{Ric_g(\mathcal{N},\mathcal{N})}+|A|^2)\phi,\phi\rangle_{L^2(N,g_N)}
\end{aligned}
$$ and consequently by multiplying both sides with $\frac{n-2}{2(n-1)}$ we get $$0\leq \langle d\phi,d\phi\rangle_{L^2\Omega^1(N,g_N)}-\langle q\phi,\phi\rangle_{L^2(N,g_N)}.$$ Now using the minimality and the Gauss-Codazzi equation we have 
{ \begin{equation}\label{GC}
\mathrm{Ric}_g(\mathcal{N},\mathcal{N})+|A|^2=\frac{1}{2}(-s_{g_N}+s_g+|A|^2).
\end{equation}}
Thus we get $$\frac{n-2}{4(n-1)}s_{g_N}+q=\frac{n-2}{4(n-1)}s_{g_N}+\frac{n-2}{4(n-1)}(-s_{g_N}+s_g+|A|^2)=\frac{n-2}{4(n-1)}(s_g+|A|^2)\geq 0.$$
The conclusion now follows by applying Th. \ref{finiteness} to $\tilde{L}$.
\end{proof}

Endowed with the previous proposition we can show that the conformal class $[g_N]$ always contains a scalar-flat metric and, under additional assumptions, it contains also a positive scalar curvature metric. More precisely we have:

\begin{prop}
\label{scalarflat}
In the setting of Prop. \ref{conformal} assume that $m>3$. Then $[g_N]$, the conformal class of $g_N$, contains a (possibly incomplete) scalar-flat metric. If $s_g|_{f(N)}$ is uniformly positive then $[g_N]$ contains a (possibly incomplete) metric with positive scalar curvature.
\end{prop}

\begin{proof}
By Prop. \ref{conformal} we know that $L$ is non-negative on $C_{c}^{\infty}(N)$. Then for every open set $\Omega\subset N$ we have $$0\leq \inf\{\langle Lu,u\rangle_{L^2(\Omega,g_N|_{\Omega})},\ u\in C_c^{\infty}(\Omega),\ \|u\|_{L^2(\Omega,g_N|_{\Omega})}^2=1\}.$$ We can thus apply \cite[Th. 1]{Fischer} (see also \cite[Lemma 3.10]{Pigola}) to conclude that there exists a smooth function $\psi$ such  that  $\psi(x)>0$ for each $x\in N$ and $L\psi=0$.  Let $\gamma:= \psi^{4/(n-2)}$. Then the scalar curvature of the metric $\tilde{g}:=\gamma g_N$ is given by $$s_{\tilde{g}}=\psi^{-\frac{n+2}{n-2}}L\psi=0.$$
Consider now the case where $s_g|_{f(N)}$ is uniformly positive. This means that there exists a positive constant $c$ such that $s_g(f(x))\geq c$ for each $x\in N$. Arguing as in the proof of Prop. \ref{conformal} we can show easily that $\langle \tilde{L}\phi, \phi\rangle_{L^2(N,g_N)}\geq c\|\phi\|^2_{L^2(N,g_N)}$ for every $\phi\in C_c^{\infty}(N)$ and thus $$\langle L\phi, \phi\rangle_{L^2(N,g_N)}\geq 4c\frac{n-1}{n-2}\|\phi\|^2_{L^2(N,g_N)}$$ for every $\phi\in C_c^{\infty}(N)$. Let $\lambda\in (0,4c\frac{n-1}{n-2})$ be arbitrarily fixed. Then the operator $\hat{L}:=L-\lambda$ satisfies  $$\langle \hat{L}\phi, \phi\rangle_{L^2(N,g_N)}\geq \left(4c\frac{n-1}{n-2}-\lambda\right)\|\phi\|^2_{L^2(N,g_N)}$$ for every $\phi\in C_c^{\infty}(N)$. Then, by repeating the above argument, we have $$0\leq \inf\{\langle \hat{L}u,u\rangle_{L^2(\Omega,g_N|_{\Omega})},\ u\in C_c^{\infty}(\Omega),\ \|u\|_{L^2(\Omega,g_N|_{\Omega})}^2=1\}$$ for every open set $\Omega\subset N$. We can thus apply again \cite[Th. 1]{Fischer} or \cite[Lemma 3.10]{Pigola}) to conclude that there exists a smooth function $\beta$ such  that  $\beta(x)>0$ for each $x\in N$ and $\hat{L}\beta=0$, that is $L\beta=\lambda \beta$.  Let $\upsilon:= \beta^{4/(n-2)}$. Then the scalar curvature of the metric $\tilde{g}:=\upsilon g_N$ is given by $$s_{\tilde{g}}=\beta^{-\frac{n+2}{n-2}}L\beta=\beta^{-\frac{n+2}{n-2}}\lambda\beta$$ which is a positive function on $N$.
\end{proof}

We collect now various applications of Th. \ref{c-length}. We are particularly interested in two Dirac operators that appear naturally in many problems of Riemannian geometry and geometric analysis: the spin-Dirac operator $\eth$ and the Hodge-de Rham operator $d+d^t$. We start with the first of the two and as a first step we provide a very succinct introduction to the spin-Dirac operator. We invite the unfamiliar reader to consult \cite{Lawson} for the background material in spin geometry.\\

A Riemannian spin structure over an oriented Riemannian manifold $(M,g)$ of dimension $m$ is a pair $(P_{\mathrm{spin}}(M),\overline{\rho})$, consisting of 
\begin{enumerate}
\item a $\mathrm{Spin}(m)$-principal bundle $P_{\mathrm{Spin}}(M)\rightarrow M$;
\item a twofold covering $\overline{\rho}:P_{\mathrm{Spin}}(M)\rightarrow P_{\mathrm{SO}}(M)$ such that the diagram 
$$\begin{tikzcd} 
P_{\mathrm{Spin}}(M)\times \mathrm{Spin}(m)\arrow{d}{(\overline{\rho},\rho)} \arrow{r} & P_{\mathrm{Spin}}(M) \arrow{d}{\rho} \arrow{dr}\\
P_{\mathrm{SO}}(M)\times \mathrm{SO}(m)\arrow{r} & P_{\mathrm{SO}}(M)
 \arrow{r} &M 
\end{tikzcd}$$
commutes. 
\end{enumerate}
Note that in the above diagram $P_{SO}(M)$ denotes the principal $SO(m)$-bundle over $M$ induced by $g$, the horizontal maps $P_{\mathrm{Spin}}(M)\times \mathrm{Spin}(m) \rightarrow P_{\mathrm{Spin}}(M)$ and $
P_{\mathrm{SO}}(M)\times \mathrm{SO}(m)\rightarrow  P_{\mathrm{SO}}(M)$ are the right group actions of $\mathrm{Spin}(m)$ and $SO(m)$, respectively, and finally the map $\rho: \mathrm{Spin}(m)\rightarrow SO(m)$ is the twofold covering homomorphism which fits into the short exact sequence of groups $$1\rightarrow \{1,-1\}\rightarrow \mathrm{Spin}(m)\stackrel{\rho}{\rightarrow} SO(m)\rightarrow 1.$$ We refer to \cite[Th. 2.10]{Lawson} for a precise definition of $\rho$. Riemannian spin structures may fail to exist or to be unique. However their existence and uniqueness are purely topological questions; more precisely they exist if and only if $w_2(M)=0$, that is the second Stiefel-Whitney class of $M$ vanishes and, if $w_2(M)=0$, the distinct Riemannian spin structures on $(M,g)$ are in one-to-one correspondence with the elements of $H^1(M,\mathbb{Z}_2)$, see i.e. \cite[Th. 1.7]{Lawson}.
\begin{defi}
An oriented Riemannian manifold $(M,g)$ is said to be spinnable if it admits Riemannian spin structures. A Riemannian spin manifold is a spinnable Riemannian manifold $(M,g)$ with a fixed Riemannian spin structure $P_{\mathrm{Spin}(m)}(M)\rightarrow M$.
\end{defi}
Let us now consider a Riemannian spin manifold $(M,g)$ and let $P_{\mathrm{Spin}}(M)\rightarrow M$ be a fixed Riemannian spin structure on $M$. Using the unitary representation of the group $\mathrm{Spin}(m)$ into the spinor space, the principal bundle $P_{\mathrm{Spin}}(M)\rightarrow M$ canonically induces a complex Hermitian vector bundle $(\Sigma M,\tau)\rightarrow M$, which is called the spinor bundle. Moreover the Levi-Civita connection of $TM$ lifts to a Hermitian connection on $\Sigma M$, $$\nabla^{\Sigma}:C^{\infty}(M,\Sigma M)\rightarrow C^{\infty}(M,T^*M\otimes \Sigma M)$$ called the spinor connection. The spin-Dirac operator $$\eth:C^{\infty}(M,\Sigma M)\rightarrow C^{\infty}(M,\Sigma M)$$ is the first order, formally self-adjoint, elliptic differential operator obtained by composing the spinor connection $\nabla^{\Sigma}$ with the isomorphism induced by $g$ $$T^*M\otimes \Sigma M\cong TM \otimes \Sigma$$ and the Clifford multiplication $$C^{\infty}(M,TM\otimes \Sigma M)\rightarrow C^{\infty}(M,\Sigma M)$$ see \cite[Ch. II]{Lawson} for details. Locally $\eth$ has the following description: let $U\subset M$ be an open neighbourhood such that $TM|_U$ is trivial and let $\{e_1,...,e_m\}$ be an orthonormal frame for $TM|_U$. Given any $s\in \Sigma M|_U$ we have $$\eth s=\sum_{k=1}^me_k\cdot \nabla^{\Sigma}_{e_k}s$$ with $\cdot$ denoting the Clifford multiplication, see \cite[Ch II, \S 5]{Lawson}.  The space of $L^2$-harmonic spinors of the spinor bundle $(\Sigma,\tau)\rightarrow M$ is defined as $$\ker(\eth)\cap L^2(M,\Sigma M,g,\tau).$$ At this point we can give the following
\begin{defi}
We say that a spinnable Riemannian manifold $(M,g)$ carries no non-trivial $L^2$-harmonic spinors if for every arbitrarily fixed Riemannian spin structure on $M$ the corresponding spinor bundle carries no non-trivial $L^2$-harmonic spinors.
\end{defi}

We have now all the ingredients for the next

\begin{teo}
\label{spin}
Let $(M,g)$ be a Riemannian manifold with $s_g\geq 0$. Let $N$ be a spinnable manifold with $\dim(N)+1=\dim(M)$ such that there exists a two-sided, stable minimal immersion $$f:N\rightarrow (M,g)$$ with $(N,g_N)$  complete. Let $P_{\mathrm{Spin(n)}}(N)\rightarrow N$ be an arbitrarily fixed Riemannan spin structure on $N$ and let $(\Sigma N,\tau)\rightarrow N$ and $\eth$ be the corresponding spinor bundle and spin-Dirac operator. Then every $L^2$-harmonic spinor of $(\Sigma N, \tau)\rightarrow N$ has constant length and consequently 
$$\dim\left(\ker(\eth)\cap L^2(N,\Sigma N,g_N,\tau)\right)\leq 2^{n/2}$$ if $n$ is even whereas
$$\dim\left(\ker(\eth)\cap L^2(N,\Sigma N,g_N,\tau)\right)\leq 2^{(n-1)/2}$$ if $n$ is odd.
\end{teo}

\begin{proof}
Let $(\Sigma N, \tau)\rightarrow N$ be the spinor bundle of $(N,g_N)$ with respect to an arbitrarily fixed Riemannian spin structure and let $\eth:C^{\infty}(N,\Sigma N)\rightarrow C^{\infty}(N,\Sigma N)$ be the corresponding spin-Dirac operator. Let $\Delta^2_{\eth}:=\eth\circ \eth$ be the spin Laplacian and let $$\nabla^{\Sigma,t}:C^{\infty}(N,T^*N\otimes \Sigma N)\rightarrow C^{\infty}(N,\Sigma N)$$ be the formal adjoint w.r.t. $g$ and $\tau$ of the spinor connection $\nabla^{\Sigma}$. The Lichnerowicz formula  \cite[Th. 8.8]{Lawson} tells us that $$\Delta^2_{\eth}=\nabla^{\Sigma,t}\circ \nabla^{\Sigma}+\frac{1}{4}s_{g_N}$$ with  $s_{g_N}$ the scalar curvature of $(N,g_N)$. By arguing as in the proof of Prop. \ref{conformal}  we get that the second inequality in \eqref{sign} holds true for $q=\frac{\mathrm{Ric}_g(\mathcal{N},\mathcal{N})+|A|^2}{2}$, namely $$\Delta-\frac{1}{2}(\mathrm{Ric}_g(\mathcal{N},\mathcal{N})+|A|^2)$$ is non-negative on $C_c^{\infty}(N)$. Concerning the first  inequality in \eqref{sign}, {by \eqref{GC}},  we have 

\begin{align}
\label{wer}
\frac{1}{4}s_{g_N}+\frac{1}{2}(\mathrm{Ric}_g(\mathcal{N},\mathcal{N})+|A|^2)&=\frac{1}{4}s_{g_N}+\frac{1}{4}(s_g+|A|^2-s_{g_N})\\
\nonumber &=\frac{1}{4}(s_g+|A|^2)\geq 0 
\end{align}
 as $s_g\geq 0$. The conclusion now follows from Th. \ref{c-length} and the fact that the rank of $\Sigma N$ is $2^{n/2}$ if $n$ is even and $2^{(n-1)/2}$ if $n$ is odd, see e.g \cite[Prop. 1.2.1]{Gino}.
\end{proof}

\begin{cor}
\label{imm2}
In the setting of Th. \ref{spin} if in addition $(N,g_N)$ has infinite volume  then $(N,g_N)$ carries no non-trivial $L^2$-harmonic spinors. In particular this is the case if there exists a constant $0\neq \delta\in \mathbb{R}$ such that $s_g\geq \delta^2>0$.
\end{cor}

\begin{proof}
The first statement follows immediately from Cor. \ref{immediately}. If $s_g\geq \delta^2>0$ then we have $$\frac{1}{4}s_{g_N}+\frac{1}{2}(\mathrm{Ric}_g(\mathcal{N},\mathcal{N})+|A|^2)=\frac{1}{4}(s_g+|A|^2)\geq \frac{\delta^2}{4}>0.$$ Therefore by the fourth point of Cor. \ref{immediately} we can conclude that $\vol_{g_N}(N)=+\infty$.
\end{proof}

Before stating the next result we recall some further properties of spin geometry. Let $N$ be an oriented manifold of dimension $n$ and let $(M,g)$ be a Riemannian spinnable manifold of dimension $m=n+1$. If there exists an immersion $f:N\rightarrow M$ then $N$ is also spinnable, see e.g. \cite[p. 20]{Gino}. Since $(\mathbb{R}^{n+1},g_{e})$, with $g_{e}$ denoting the standard Euclidean metric, is a Riemannian spinnable manifold, by the above property we can deduce that every oriented $n$-dimensional immersed manifold $f:N\rightarrow (\mathbb{R}^{n+1},g_{e})$ is also spinnable. We have now all the ingredients to prove the following {\em spinorial version of Tanno's problem}:

\begin{teo}\label{spin-cor}
Let $N$ be an oriented manifold with $\dim(N)=n$. If there exists a stable minimal immersion $$f:N\rightarrow (\mathbb{R}^{n+1},g_e)$$ such that $(N,g_N)$ is complete then $(N,g_N)$ carries no non-trivial $L^2$-harmonic spinors.
\end{teo}

\begin{proof}
Thanks to the above observation we know that $N$ is spinnable. Since $f:N\rightarrow (\mathbb{R}^{n+1},g_e)$ is minimal  and $(N,g_N)$ is complete we know that $(N,g_N)$ has infinite volume. This is well known and for instance follows easily by the fact that $(N,g_N)$ carries a Sobolev embedding, see \cite{Michael}, and the local volume estimate for complete manifold supporting a Sobolev embedding, see \cite[Lemma 2.2]{Hebey}.  Now the conclusion follows immediately by Cor. \ref{imm2}.
\end{proof}

\begin{cor}
\label{s-flat}
Let $(M,g)$ be a Riemannian manifold with $s_g\geq 0$. Let $N$ be spinnable manifold with $\dim(N)+1=\dim(M)$ such that there exists a two-sided, stable, minimal immersion $f:N\rightarrow (M,g).$ If $N$ carries a non-trivial $L^2$-harmonic spinor then $f:N\rightarrow (M,g)$ is totally geodesic and 
\begin{equation}
\label{RicN}s_{g_N}=-2\mathrm{Ric}_g(\mathcal{N},\mathcal{N}).
\end{equation}
\end{cor}

\begin{proof}
By the third point of Cor. \ref{immediately} and \eqref{wer} we get that $|A|$ vanishes  and thus $f:N\rightarrow (M,g)$ is totally geodesic. For the same reason $s_g|_N$ vanishes. Now the Gauss-Codazzi equation for the scalar curvature tells us that 
$$s_{g_N}=-2\mathrm{Ric}_g(\mathcal{N},\mathcal{N})$$
as required.
\end{proof}

\begin{cor}
\label{s-flat}
Let $(M,g)$ be a Riemannian manifold with $s_g\geq 0$. Let $N$ be a compact spin manifold with $\dim(N)+1=\dim(M)$ such that there exists a two-sided, stable minimal immersion $f:N\rightarrow (M,g).$
\begin{enumerate}
\item If $n$ is even the $\hat{A}$-genus of $N$ satisfies the upper bound $$|\hat{A}(N)|\leq 2^{\frac{n-2}{2}}.$$ If in addition $\pi_1(N)$ is infinite then $$\hat{A}(N)=0.$$ 
\item If $\hat{A}(N)\neq 0$ and $\mathrm{Ric}_g(\mathcal{N},\mathcal{N})\geq 0$  then $(N,g_N)$ is Ricci-flat with $\pi_1(N)$ finite and $H^1(N,\mathbb{R})=\{0\}$. Moreover $(M,g)$ is also Ricci-flat.
\end{enumerate}
\end{cor}

\begin{proof}
First of all we recall that $\hat{A}(N)$, the $\hat{A}$-genus of $N$, is a characteristic number of $N$ defined as the integral over $N$ of a certain combination of Pontryagin classes, see \cite[p. 138]{Lawson} for a precise definition.  The first assertion is clearly an immediate consequence of the celebrated   Atiyah-Singer index theorem for the spin-Dirac operator. We give a brief comment about it and we refer to \cite{Lawson} for details. Let $(\Sigma N,\tau)\rightarrow N$ be an arbitrarily fixed spinor bundle over $N$. Then $\Sigma N$ decomposes as $\Sigma N=\Sigma N^+\oplus \Sigma N^-$, with $\Sigma N^+$ and $ \Sigma N^-$ the spinor bundle of positive/negative chirality, and the spin-Dirac operator $\eth$ satisfies $$\eth|_{C^{\infty}(N,\Sigma N^{\pm})}:C^{\infty}(N,\Sigma N^{\pm})\rightarrow C^{\infty}(N,\Sigma N^{\mp}).$$ The Atiyah-Singer index theorem then shows that $$\hat{A}(N)=\dim\left(\ker(\eth|_{C^{\infty}(N,\Sigma N^+)})\right)-\dim\left(\ker(\eth|_{C^{\infty}(N,\Sigma N^-)})\right).$$
Consequently we have $$\hat{A}(N)\leq \dim\left(\ker(\eth|_{C^{\infty}(N,\Sigma N^+)})\right)\quad \mathrm{and}\quad  -\hat{A}(N)\leq \dim\left(\ker(\eth|_{C^{\infty}(N,\Sigma N^-)})\right).$$ Snce $\Sigma N^+$ and $\Sigma N^-$ are both complex vector bundles of complex rank $2^{\frac{n-2}{2}}$, we can conclude by Cor. \ref{immediately} that $|\hat{A}(N)|\leq 2^{\frac{n-2}{2}}$. 
Assume now that $\pi_1(N)$ is infinite. Let $\tilde{N}\rightarrow N$ be the universal covering of $N$ and let $\tilde{f}:\tilde{N}\rightarrow M$ be the lift of $f$. According to \cite[p. 46]{ColMin} we know that $\tilde{f}:\tilde{N}\rightarrow (M,g)$ is still a stable minimal immersion. Therefore, as $\tilde{N}$ is still spinnable, we can apply Cor. \ref{imm2} to conclude that there are no non-trivial $L^2$-harmonic spinors on $(\tilde{N},g_{\tilde{N}})$. Finally by Atiyah's $L^2$-index theorem \cite[Th. 3.8]{Atiyah}, we can conclude that $\hat{A}(N)=0$. We now deal with the { second} point. Since $\hat{A}(N)\neq 0$ we know that on $(N,g_N)$ there exists a  non-trivial harmonic spinor $s$. Note now that since $N$ is compact we can pick the constant sequence of functions $\phi_n=1$ in the proof of Th. \ref{c-length}. Thus \eqref{zerozeroo} gives in this case
$$\int_{N}(|d|s|_{\rho}|_g^2-(\mathrm{Ric}_g(\mathcal{N},\mathcal{N})+|A|^2)|s|^2_{\rho})\dvol_g=0,\quad\quad \int_N(\rho((\mathrm{Ric}_g(\mathcal{N},\mathcal{N})+|A|^2)s,s))\dvol_g=0.$$ By Th. \ref{spin} we know that $|s|_{\rho}$ is constant and by Cor. \ref{s-flat} we know that $|A|$ vanishes. Therefore the first equality above boils down to $$\int_{N}(\mathrm{Ric}_g(\mathcal{N},\mathcal{N}))|s|^2_{\rho})\dvol_g=0$$ and since $|s|_{\rho}\neq 0$ and $\mathrm{Ric}_g(\mathcal{N},\mathcal{N})\geq 0$ we can conclude that $\mathrm{Ric}_g(\mathcal{N},\mathcal{N})=0$. Hence by \eqref{RicN} we obtain that $(N,g_N)$ is scalar flat. Furthermore we know that $N$ carries no Riemannian metrics with positive scalar curvature since $\hat{A}(N)\neq 0$, see \cite[Cor. 8.9]{Lawson}. Therefore by \cite[Th. 2.30]{Lee} we can conclude that $(N,g_N)$ is actually Ricci-flat. Moreover by the first statement of this corollary we know that $\pi_1(N)$ is finite and by  \cite[Th. 4.1]{FWolf} we obtain that $H^1(N,\mathbb{R})=\{0\}$. Finally, according to \cite{SchoenYau}, if $M$ carries a metric with positive scalar curvature then also $N$ carries a metric with positive scalar curvature. However this is not possible since we assumed that $\hat{A}(N)\neq 0$. Therefore, by applying again  \cite[Th. 2.30]{Lee}, we can conclude that $(M,g)$ is Ricci flat, as well.
\end{proof}

We revolve now our attention to the vanishing of $L^2$-harmonic forms. Given a {\em complete} Riemannian manifold $(N,g)$ we recall that the space of $L^2$-harmonic forms of degree $k$ is defined as $$\mathcal{H}^k_2(N,g):=\ker(\Delta_k)\cap L^2\Omega^k(N,g)$$ where $$\Delta_k:\Omega^k(N,g)\rightarrow \Omega^k(N,g)$$ denotes the Hodge Laplacian. Since $$\Delta_k:L^2\Omega^k(N,g)\rightarrow L^2\Omega^k(N,g)$$ with initial domain $\Omega_c^k(N)$ is essentially self-adjoint, we get that 
\begin{equation}
\label{equivalent}
\mathcal{H}^k_2(N,g)=\ker(d_k)\cap \ker(d^t_{k-1})\cap L^2\Omega^k(N,g)
\end{equation}
with $d^t_{k-1}:\Omega^{k}(N)\rightarrow \Omega^{k-1}(N)$ the formal adjoint of $d_{k-1}:\Omega^{k-1}(N)\rightarrow \Omega^k(N)$ w.r.t. $g$. Let us consider now the Dirac operator $$d+d^t:\Omega^{\bullet}(N)\rightarrow \Omega^{\bullet}(N)$$ with $\Omega^{\bullet}(N)=\oplus_k\Omega^{k}(N)$ and $d+d^t|_{\Omega^k(N)}=d_k+d_{k-1}^t$, usually called the Hodge-de Rham operator or the Gauss-Bonnet operator. Then \eqref{equivalent} can be reformulated as $$\mathcal{H}^k_2(N,g)=\ker(d+d^t)\cap L^2\Omega^k(N,g).$$ The advantage of bringing $d+d^t$ into the description of the space of $L^2$-harmonic forms is that now we can use Th. \ref{c-length} to deal with their vanishing over a two-sided, stable, minimally immersed hypersurface. More precisely let $(M,g)$ be a Riemannian manifold of dimension $m+1$ and let $N$ be an oriented manifold of dimension $m$ which carries a two-sided, stable, minimal immersion $f:N\rightarrow (M,g)$ such that $(N,g_N)$ is complete. The vanishing of $\mathcal{H}^0_2(N,g_N)$ holds if and only if $(N,g_N)$ has infinite volume while the vanishing of $\mathcal{H}^1_2(N,g_N)$ has been already extensively studied, see e.g. \cite{Miyaoka}, \cite{Palmer} and \cite{Tanno}. In particular it is known that $\mathcal{H}^1_2(N,g_N)=\{0\}$, provided $(M,g)$ has positive bi-Ricci curvature, see \cite[Th. 4.1]{Tanno}. In this paper we are therefore interested in the behaviour of $\mathcal{H}^k_2(N,g_N)$ when $k\geq 2$. Before stating our main result we need to introduce some notation and recall some preliminary results from \cite{Savo}. Let $(M,g)$ be a Riemannian manifold of dimension $m$ and let $f:N\rightarrow M$ be an immersed submanifold of dimension $n$. According to the Weitzenb\"ock formula the Hodge Laplacian on $(N,g_N)$ acting  on $p$-forms decomposes as $$\Delta_p=\nabla^t\circ \nabla+\mathcal{B}_p$$ with $\nabla:\Omega^p(N)\rightarrow C^{\infty}(N,T^*N\otimes \Lambda^kT^*N)$ the connection induced by the Levi-Civita connection and $\mathcal{B}_p$ a suitable endomorphism of $\Lambda^kT^*M$ that can be expressed in term of the curvature operator of $(N,g_N)$ and the Clifford multiplication, see e.g. \cite[Th. 50]{Peter}. According to \cite[Th. 1]{Savo} $\mathcal{B}_p$ can be decomposed as $\mathcal{B}_p=\mathcal{B}_{p,\mathrm{ext}}+\mathcal{B}_{p,\mathrm{res}}$, with $\mathcal{B}_{p,\mathrm{ext}}$ and $\mathcal{B}_{p,\mathrm{res}}$ two further endomorphisms of $\Lambda^kT^*N$ with the following properties: $\mathcal{B}_{p,\mathrm{res}}$ depends only on $\mathcal{R}_g$, the curvature operator of $(M,g)$, and if $\mathcal{R}_g\geq \gamma$ with $\gamma\in \mathbb{R}$ then  $\mathcal{B}_{p,\mathrm{res}}\geq p(n-p)\gamma$. Conversely
$\mathcal{B}_{p,\mathrm{ext}}$ depends only on the second fundamental form of $f:N\rightarrow (M,g)$. In particular if we denote with $k_1,...,k_n$ the principal curvatures of $N$, then we have the estimate 
\begin{equation}
\label{pcc}
\mathcal{B}_{p,\mathrm{ext}}\geq \min_{\alpha\subseteq \{1,...,n\},\ |\alpha|=p} K_{\alpha}K_{\star\alpha}
\end{equation}
with $\alpha=\{j_1,...,j_p\}\subseteq \{1,...,n\}$, $K_{\alpha}=k_{j_1}+...+k_{j_p}$ and $K_{\star\alpha}$  defined analogously with respect to $\star\alpha:=\{1,...,n\}\setminus \alpha$. We are now in the position to state the next theorem.

\begin{teo}
\label{noL2forms}
Let $(M,g)$ be a Riemannian manifold of dimension $m+1$ and let $\Sigma$ be an oriented manifold of dimension $m$ with a two-sided, stable, minimal immersion $f:\Sigma\rightarrow (M,g)$ such that $(\Sigma,g_{\Sigma})$ is complete. { Let $2\leq p\leq \frac m2$}.
\begin{enumerate}
\item  If $m\geq 4$ and $\mathcal{R}_g\geq0$ (or more generally there exists $\gamma\in \mathbb{R}$ such that $\mathcal{R}_g\geq \gamma$ and $p(m-p)\gamma+\mathrm{Ric}_g(\mathcal{N},\mathcal{N})\geq 0$) and  $|A|^2-K^2_{\alpha}\geq 0$ for any $\alpha\subset \{1,...,m\}$ with $|\alpha|=p$, then every $L^2$-harmonic $p$-form on $(\Sigma,g_\Sigma)$ has constant length.  If in addition $\Sigma$ is not totally geodesic or $\mathrm{Ric}_g(\mathcal{N},\mathcal{N})$ is somewhere positive on $f(\Sigma)$, then $$\mathcal{H}^p_2(\Sigma,g_{\Sigma})=\mathcal{H}^{m-p}_2(\Sigma,g_{\Sigma})=\{0\}.$$
\item If $m\geq 6$ and $\mathrm{sec}_g\in [a,b]$ with $0<a\leq b\leq \varepsilon_{m,p}a$ with $\varepsilon_{m,p}$ the constant defined in \eqref{nonsharp} and $|A|^2-K^2_{\alpha}\geq 0$ for any $\alpha\subset \{1,...,m\}$ with $|\alpha|=p$, then $$\mathcal{H}^p_2(\Sigma,g_{\Sigma})=\mathcal{H}^{m-p}_2(\Sigma,g_{\Sigma})=\{0\}.$$
\item If $m\geq 6$ and $\mathrm{sec}_g\in [a,b]$ with $0<a\leq b\leq c_ma$, $c_m$ the constant defined in \eqref{sharp} and $|A|^2-K^2_{\alpha}\geq 0$ for any $\alpha\subset \{1,...,m\}$ with $|\alpha|=2$, then $$\mathcal{H}^2_2(\Sigma,g_{\Sigma})=\mathcal{H}^{m-2}_2(\Sigma,g_{\Sigma})=\{0\}.$$
\end{enumerate}
\end{teo}

\begin{proof}
First of all we point out that thanks to the $L^2$-Poincar\'e duality it is enough to prove that $\mathcal{H}^p_2(\Sigma,g_{\Sigma})=\{0\}$. Consider now the Weitzenb\"ock formula  for the Hodge Laplacian on $(\Sigma,g_\Sigma)$ acting  on $p$-forms  $$\Delta_p=\nabla^t\circ \nabla+\mathcal{B}_p.$$ 
Since $f:\Sigma\rightarrow (M,g)$ is two-sided and stable we know that the second inequality in \eqref{sign} holds true for $q=\mathrm{Ric}_g(\mathcal{N},\mathcal{N})+|A|^2$, that is $$\Delta-\mathrm{Ric}_g(\mathcal{N},\mathcal{N})-|A|^2$$ is non-negative on $C_c^{\infty}(\Sigma)$. Concerning the first  inequality in \eqref{sign}  we have $$\mathcal{B}_p+\mathrm{Ric}_g(\mathcal{N},\mathcal{N})+|A|^2=\mathcal{B}_{p,\mathrm{res}}+\mathcal{B}_{p,\mathrm{ext}}+\mathrm{Ric}_g(\mathcal{N},\mathcal{N})+|A|^2.$$
As recalled above, see \cite[Th. 1]{Savo}, if there exists $\gamma\in \mathbb{R}$ such that $\mathcal{R}_g\geq \gamma$ then $\mathcal{B}_{p,\mathrm{res}}\geq p(m-p)\gamma$, see also \cite{Gallot}. Therefore we obtain $$\mathcal{B}_{p,\mathrm{res}}+\mathrm{Ric}_g(\mathcal{N},\mathcal{N})\geq \gamma p(m-p)+\mathrm{Ric}_g(\mathcal{N},\mathcal{N})\geq 0.$$ In particular $\mathcal{B}_{p,\mathrm{res}}+\mathrm{Ric}_g(\mathcal{N},\mathcal{N})\geq0$ whenever $\mathcal{R}_g\geq 0$. Indeed in this case we get that both $$\mathrm{Ric}_g(\mathcal{N},\mathcal{N})\geq 0\quad \mathrm{and}\quad \mathcal{B}_{p,\mathrm{res}}\geq 0.$$ 
For the other term we have 
\begin{equation}
\label{ext}
\mathcal{B}_{p,\mathrm{ext}}+|A|^2\geq |A|^2+\min_{\alpha\subseteq \{1,...,m\},\ |\alpha|=p} K_{\alpha}K_{\star\alpha}\geq |A|^2+\min_{\alpha\subseteq \{1,...,m\},\ |\alpha|=p} -K^2_{\alpha}\geq 0
\end{equation}
as $K_{\star \alpha}=-K_{\alpha}$ given that $f:\Sigma\rightarrow (M,g)$ is minimal and $|A|^2-K^2_{\alpha}\geq 0$ for any $\alpha\subset \{1,...,m\}$ with $|\alpha|=p$, by assumptions. Thus by Th. \ref{c-length} we can conclude that every $L^2$-harmonic $p$-form on $(\Sigma,g_{\Sigma})$ has constant length. Note now that the inequality $\gamma p(m-p)+\mathrm{Ric}_g(\mathcal{N},\mathcal{N})\geq 0$ implies  $\mathrm{Ric}_g(\mathcal{N},\mathcal{N})\geq 0$ regardless the sign of $\gamma$. Indeed it is obvious that $
\mathrm{Ric}_g(\mathcal{N},\mathcal{N})\geq 0$ if $\gamma<0$ whereas if $\gamma\geq 0$  then $\mathcal{R}_g\geq 0$ and thus $
\mathrm{Ric}_g(\mathcal{N},\mathcal{N})\geq 0$. Therefore if  in addition we assume that $\Sigma$ is not totally geodesic or $\mathrm{Ric}_g(\mathcal{N},\mathcal{N})$ is somewhere positive on $f(\Sigma)$, we can then apply the fourth point of Cor. \ref{immediately} with $q=|A|^2+\mathrm{Ric}_g(\mathcal{N},\mathcal{N})$ to conclude that $$\vol_{g_{\Sigma}}(\Sigma)=+\infty\quad\quad \mathrm{and}\quad\quad \mathcal{H}^p_2(\Sigma,g_{\Sigma})=\{0\}.$$ 
The proof of the first point is thus complete. Let us tackle now the second point and let
\begin{equation}
\label{nonsharp}
\varepsilon_{m,p}:=\frac{p(m-p)(2\mu+1)+3m}{2p(m-p)(\mu-1)}
\end{equation}
with $\mu:=[(m+1)/2]$. As  explained above it is enough to prove that $\mathcal{H}^p_2(\Sigma,g_{\Sigma})=\{0\}$. Arguing as in \eqref{ext} we know that $\mathcal{B}_{p,\mathrm{ext}}+|A|^2\geq 0$. For the remaining term, $\mathcal{B}_{p,\mathrm{res}}+\mathrm{Ric}_g(\mathcal{N},\mathcal{N})$, we have
$$\mathcal{B}_{p,\mathrm{res}}+\mathrm{Ric}_g(\mathcal{N},\mathcal{N})\geq p(m-p)\left(\frac{a+b}{2}-\frac{b-a}{6}(4\mu-1)\right)+ma\geq 0$$ 
where the first inequality follows by applying first \cite[Prop. 3.8]{BoKa} and then \cite[Th. 1]{Savo} and the second one by the pinching condition $a\leq b\leq \varepsilon_{m,p}a$. We can thus conclude that  $$\mathcal{B}_p+\mathrm{Ric}_g(\mathcal{N},\mathcal{N})+|A|^2\geq 0$$ and so each $L^2$-harmonic $p$-form on $(\Sigma,g_{\Sigma})$ has constant length. Finally note that $|A|^2+\mathrm{Ric}_g(\mathcal{N},\mathcal{N})$ is a positive function on $\Sigma$, given that $(M,g)$ is positively pinched. Therefore by the fourth point of Cor. \ref{immediately} we can conclude that $\mathcal{H}^p_2(\Sigma,g_{\Sigma})=\{0\}$. Let us tackle now the third point. By arguing as in the proof of the second point it suffices to show that $\mathcal{B}_2+|A|^2+\mathrm{Ric}_g(\mathcal{N},\mathcal{N})\geq 0$. Let 
\begin{equation}
\label{sharp}
c_m:=\left\{
{ \begin{array}{lll} 
\frac{(11m-16)}{2(m-2)} &  \mathrm{if}\ m\ \mathrm{is\ even},\\
\\
\frac{(11m-18)}{2(m-3)} & \mathrm{if}\ m\ \mathrm{is\ odd}
\end{array}}
\right.\end{equation} By \eqref{ext} we know  that $\mathcal{B}_{2,\mathrm{ext}}+|A|^2\geq0$. In order to conclude we are left to prove that $$\mathcal{B}_{2,\mathrm{res}}+\mathrm{Ric}_g(\mathcal{N},\mathcal{N})\geq 0.$$ Let us consider an arbitrarily fixed $\omega\in \Omega^2(\Sigma)$. According to \cite[Lemma 2.2]{Zhu2} given any point $p\in \Sigma$ there
exists an open neighbourhood $U$ and an orthonormal trivialization $\{e_1,...,e_m\}$ of $T\Sigma|_U$ such that 
\begin{equation}
\label{localform}
\omega=\sum_{i=1}^{\ell}\alpha_i\theta^{2i-1}\wedge \theta^{2i}=\frac{1}{2}\sum_{i=1}^{\ell}\alpha_i(\theta^{2i-1}\wedge \theta^{2i}-\theta^{2i}\wedge \theta^{2i-1})
\end{equation}
where $2\ell\leq m$, $\{\theta_1,...,\theta_m\}$ is the dual basis of $\{e_1,...,e_m\}$ and $\alpha_i$ are smooth functions on $U$. Therefore by writing $$\omega=\sum_{r,s=1}^{\ell} \alpha_{r,s}\theta^{r}\wedge \theta^{s}$$ with $\alpha_{r,s}$ smooth functions on $U$, we get 
\begin{equation}\label{defalpha}\alpha_{r,s}=\left\{
\begin{array}{lll} 
\frac{\alpha_i}{2} & r=2i-1, s=2i \\
-\frac{\alpha_i}{2} & r=2i, s=2i-1\\
0 & \mathrm{otherwise}
\end{array}
\right.\end{equation}
Now, with a little abuse of notation, let us still denote with $g_{\Sigma}$ the metric on $\Lambda^2T^*\Sigma$ induced by $g_{\Sigma}$. By \cite[(2.2)]{Zhu}, see also \cite[\S 3]{Li}, we have $$g_{\Sigma}(\mathcal{B}_2\omega,\omega)=2\mathrm{Ric}_{i,j}\alpha_{i,l}\alpha_{j,l}+R_{i,j,k,l}\alpha_{l,i}\alpha_{k,j}$$ where in the above equality we used  Einstein convention about repeated indices. Therefore by \eqref{defalpha} we obtain
$$
\begin{aligned}
g_{\Sigma}(\mathcal{B}_2\omega,\omega)&=\sum_{i=1}^{\ell}(\mathrm{Ric}_{2i-1,2i-1}+ \mathrm{Ric}_{2i,2i})\alpha_i^2\\
&+ \frac{1}{2}\sum_{i,j=1}^{\ell}(R_{2i,2j,2j-1,2i-1}+R_{2i-1,2j-1,2j,2i}-R_{2i,2j-1,2j,2i-1}-R_{2i-1,2j,2j-1,2i})\alpha_i\alpha_j\\
&= \frac{1}{2}\sum_{i=1}^{\ell} (\mathrm{Ric}_{2i-1,2i-1}+\mathrm{Ric}_{2i,2i}- 2R_{2i,2i-1,2i,2i-1})\alpha_i^2\\
&+ \frac{1}{2}\sum_{i\neq j=}^{\ell}(R_{2i,2j,2j-1,2i-1}+R_{2i-1,2j-1,2j,2i}-R_{2i,2j-1,2j,2i-1}-R_{2i-1,2j,2j-1,2i})\alpha_i\alpha_j\\
&=\frac{1}{2}\sum_{i=1}^{\ell}(\mathrm{Ric}_{2i-1,2i-1}+\mathrm{Ric}_{2i,2i}- 2R_{2i,2i-1,2i,2i-1})\alpha_i^2-\sum_{i\neq j=1}^{\ell}R_{2i-1,2i,2j-1,2j}\alpha_i\alpha_j
\end{aligned}
$$
where in the last equality we used the Bianchi identity of $R_{g_\Sigma}$. By means of the Gauss-Codazzi equation and keeping in mind that $\mathcal{B}_{2,\mathrm{res}}$ depends only on the curvature tensor of $(M,g)$,  we can argue as in \cite[(2.6)]{Zhu} to obtain
$$
\begin{aligned}
&g_{\Sigma}(\mathcal{B}_{2,\mathrm{res}}\omega,\omega)+\mathrm{Ric}_g(\mathcal{N},\mathcal{N})g_{\Sigma}(\omega,\omega)=\\
&\frac{1}{2}\sum_{i=1}^{\ell}\left(\sum_{k=1}^m (\overline{R}_{2i-1,k,2i-1,k}+ \overline{R}_{2i,k,2i,k})-2\overline{R}_{2i,2i-1,2i,2i-1}+\mathrm{Ric}_g(\mathcal{N},\mathcal{N})\right)\alpha_i^2-\sum_{i\neq j=1}^{\ell}\overline{
R}_{2i-1,2i,2j-1,2j}\alpha_i\alpha_j=\\
&\frac{1}{2}\sum_{i=1}^{\ell}\left(\sum_{k\leq m,k\neq 2i-1,2i}(
\overline{R}_{2i-1,k,2i-1,k}+\overline{R}_{2i,k,2i,k})+ \mathrm{Ric}_g(\mathcal{N},\mathcal{N})\right)\alpha_i^2-\sum_{i\neq j=1}^{\ell}\overline{R}_{2i-1,2i,2j-1,2j}\alpha_i\alpha_j,
\end{aligned}
$$
where we denoted with $\overline{R}_{i,j,k,l}$ the components of the curvature tensor of $(M,g)$ and in the last equality we used the symmetries of $\overline{R}_{i,j,k,l}$  and some trivial cancellation of equal terms with opposite sign. Since we assumed that $0<a\leq \mathrm{sec}_g\leq b$ we can apply { Berger's inequality \cite[(7)]{Berger} to obtain the upper bound $|\overline{R}_{i,j,k,l}|\leq \frac{2}{3}(b-a)$. Together with }Young's inequality and the assumption on the pinching of   $\sec_g$, namely that $b\leq c_ma$ with $c_m$ defined in \eqref{sharp}, we get 
$$
\begin{aligned}
&g_{\Sigma}(\mathcal{B}_{2,\mathrm{res}}\omega,\omega)+\mathrm{Ric}_{g_N}(\mathcal{N},\mathcal{N})g_{\Sigma}(\omega,\omega)=\\
&\frac{1}{2}\sum_{i=1}^{\ell}\left(\sum_{k\leq m,k\neq 2i-1,2i}(a+a)+ma\right)\alpha_i^2 -\frac{1}{2}\sum_{i\neq j=1}^{\ell}|\overline{R}_{2i-1,2i,2j-1,2j}|(\alpha_i^2+\alpha_j^2)=\\
& \frac{(3m-4)a}{2}\sum_{i=1}^{\ell}\alpha_i^2-\frac{b-a}{3}\sum_{i\neq j=1}^{\ell}(\alpha_i^2+\alpha_j^2)=\left(\frac{(3m-4)a}{2}-\frac{{ 2(\ell-1)}(b-a)}{3}\right)\sum_{i=1}^{\ell}\alpha_i^2\geq 0,
\end{aligned}
$$
given that $${ \frac{2 (\ell-1)}{3}=} \left\{
\begin{array}{lll} 
\frac{m-2}{3} &  \mathrm{if}\ m\  \mathrm{is\ even},\\
\frac{m-3}{3} & \mathrm{if}\ m\ \mathrm{is\ odd}
\end{array}
\right.$$
 and thus $$\frac{(3m-4)a}{2}-\frac{{2(\ell-1)}(b-a)}{3}= \left\{
\begin{array}{lll} 
\frac{(11m-16)a}{6}-\frac{(m-2)b}{3} &  \mathrm{if}\ m\ \mathrm{is\ even},\\
\frac{(11m-18)a}{6}-\frac{(m-2)b}{3} & \mathrm{if}\ m\ \mathrm{is\ odd}
\end{array}
\right.$$ We can thus conclude that also in this case $$\mathcal{B}_2+\mathrm{Ric}_g(\mathcal{N},\mathcal{N})+|A|^2\geq 0$$ and thus $\mathcal{H}^2_{2}(\Sigma,g_{\Sigma})=\{0\}$, as required.
\end{proof}

We have now various comments and corollaries.

\begin{rem}
In Th. \ref{noL2forms}, as well as in the subsequent results, the vanishing of $\mathcal{H}^p_2(\Sigma,g_{\Sigma})$ does not require $\Sigma$ to be oriented. The orientability of $\Sigma$ is only needed to deduce the vanishing of $\mathcal{H}^{m-p}_2(\Sigma,g_{\Sigma})$ from the vanishing of $\mathcal{H}^p_2(\Sigma,g_{\Sigma})$ through the $L^2$-Poincar\'e inequality.
\end{rem}

\begin{rem}
The definition of the curvature operator adopted in \cite[p. 74]{BoKa} produces the double of the usual curvature operator, see e.g. \cite[p. 36]{Peter} . This is why in the proof of the second point of  Th. \ref{noL2forms} appears $(\frac{a+b}{2}-\frac{b-a}{6}(4\mu-1))$ rather than $(a+b-\frac{b-a}{3}(4\mu-1))$, see \cite[Prop. 3.8]{BoKa}.
\end{rem}

\begin{rem}
 The third point of the above theorem improves \cite[Th. 1.1]{Zhu} as we used a weaker pinching condition. In particular ${ {c_m}^{-1}}< \frac{2}{11}< \frac{1}{4}$ for any value of $m$, hence a complete simply connected Riemannian manifold $(M,g)$ satisfying the pinching of the third point  of Th. \ref{noL2forms} is not necessarily diffeomorphic to a sphere, contrary to what happens in \cite{Zhu} where the pinching constant is $\frac{5}{17}$. Note also that the pinching condition used in the third point of Th. \ref{noL2forms} is weaker than the one used in second point of the same theorem for $p=2$. Moreover the pinching and the condition $\mathcal R_g\geq 0$ are independent, see Remark \ref{rmkcond} below.
\end{rem}

\begin{rem}
\label{L21}
{ Here we collect some remarks about the condition $|A|^2-K_{\alpha}^2\geq 0$.}
\begin{itemize}
\item[i)]
When $p=1$ it is immediate to verify that the condition $|A|^2-K_{\alpha}^2\geq 0$ is always satisfied. Therefore to have $\mathcal{B}_1+|A|^2+\mathrm{Ric}_g(\mathcal{N},\mathcal{N})\geq 0$ it suffices to require $\mathcal{B}_{1,\mathrm{res}}+\mathrm{Ric}_g(\mathcal{N},\mathcal{N})\geq 0$. Note that the latter inequality is an equivalent reformulation of the condition used in \cite[Th. 4.1]{Tanno}, since for each $x\in f(\Sigma)$ and $v\in T_{f(x)}M$ we have $\mathcal{B}_{1,\mathrm{res}}v+\mathrm{Ric}_g(\mathcal{N}_x,\mathcal{N}_x)=\mathrm{biRic}_g(v,\mathcal{N}_x)$. We recall that $\mathrm{biRic}_g(v,\mathcal{N}_x)$ denotes the bi-Ricci curvature along $v$ and $\mathcal{N}_x$, that is $\mathrm{biRic}_g(v,\mathcal{N}_x)=\mathrm{Ric}_g(v,v)+\mathrm{Ric}_g(\mathcal{N}_x,\mathcal{N}_x)-\mathrm{sec}_g(v,\mathcal{N}_x)$, and that if $\mathrm{sec}_g\geq 0$ then $\mathrm{biRic}_g\geq 0$.
{ \item[ii)] The above condition is always satisfied for any $p$  if $m\leq 4$, and it is in general false for $m\geq 5$ and $p\in\{2,\dots m-2\}$, see \cite[Lemma 1]{Tanno} and the subsequent proposition.
\item[iii)] Given $m$ and $p\in\{2,\dots,m-2\}$, the conditions $H=0$ and $|A|^2-K_{\alpha}^2\geq 0$ for any multi-index $\alpha$ with $|\alpha|=p$ do not necessarily implie that the hypersurface is totally geodesic:  for example $(-k,0,\dots,0,k)\in\mathbb R^m$ satisfies this condition for any $k\in\mathbb R$.
\item[iv)] The example of the previous point shows also that the above condition does not necessarily implies that $|A|^2$ is bounded. However, if $|A|^2$ satisfies a suitable bounds, the hypothesis $|A|^2-K_{\alpha}^2\geq 0$ can be dropped, see Theorem \ref{Abound} below.
}
\end{itemize}
\end{rem}

The next proposition exhibits some sufficient conditions to have $|A|^2-K^2_{\alpha}\geq 0$. More precisely:
\begin{prop}\label{zeropp}
Let $(M,g)$ be a Riemannian manifold of dimension $m+1$ and let $\Sigma$ be an oriented manifold of dimension $m\geq 4$ with a two-sided, stable, minimal immersion $f:\Sigma\rightarrow (M,g)$. We have the following properties:
\begin{enumerate}
\item If for every point $x\in \Sigma$ there are at most four non-zero principal curvatures, then $|A|^2-K_{\alpha}^2\geq0$ for each $\alpha\subset \{1,...,m\}$.
\item { if for every point $x\in \Sigma$ there are exactly two non-zero principal curvature, $\pm k$, both with the same multiplicity $2<l\leq m/2$, then $|A|^2-K_{\alpha}^2\geq 0$ for each $\alpha\subset \{1,...,m\}$ with $|\alpha|\leq \sqrt{2l}$.}
\end{enumerate}
\end{prop}

\begin{proof}{ The first point can be easily reduced to \cite[Lemma1]{Tanno}.
In the second one we have
$$
|A|^2-K_{\alpha}^2\geq(2l-|\alpha|^2)k^2\geq 0
$$
}
since we assumed  $|\alpha|\leq \sqrt{2l}$.
\end{proof}

\begin{cor}
Let $(M,g)$ be a Riemannian manifold of dimension $m+1$ and let $\Sigma$ be an oriented manifold of dimension $m\geq 4$ with a two sided, stable, minimal immersion $f:\Sigma\rightarrow (M,g)$ such that $(\Sigma,g_{\Sigma})$ is complete.
Assume that $\mathcal{R}_g\geq0$ (or more generally there exists $\gamma\in \mathbb{R}$ such that $\mathcal{R}_g\geq \gamma$ and $p(m-p)\gamma+\mathrm{Ric}_g(\mathcal{N},\mathcal{N})\geq 0$),  $|A|^2-K^2_{\alpha}\geq 0$ for any $\alpha\subset \{1,...,m\}$ with $|\alpha|=p$ and $\mathcal{H}^p_2(\Sigma,g_{\Sigma})\neq \{0\}$ { for some $p\in\{2,\dots m-2\}$}. Then $\Sigma$ is compact, totally geodesic and $\mathrm{Ric}_g(\mathcal{N},\mathcal{N})$ vanishes.
\end{cor}

\begin{proof}
Clearly the first point of Th. \ref{noL2forms} and Cor. \ref{immediately} imply that $f:\Sigma\rightarrow (M,g)$ is  totally geodesic and $\mathrm{Ric}_g(\mathcal{N},\mathcal{N})$ vanishes. Let us show that $\Sigma$ is necessarily compact. The Gauss-Codazzi equation and the assumptions on $(M,g)$ tell us that $\mathrm{Ric}_{g_{\Sigma}}\geq0$ and therefore if $\Sigma$ is not compact it has necessarily infinite volume thanks to \cite[Th. 7]{Yau}. Since every $L^2$-harmonic $p$-form has constant length this would implies that $\mathcal{H}^p_2(\Sigma,g_{\Sigma})=\{0\}$, which contradicts the assumption
$\mathcal{H}^p_2(\Sigma,g_{\Sigma})\neq\{0\}$. We can thus conclude that $\Sigma$ is compact.
\end{proof}

\begin{teo}\label{Abound}
Let $(M,g)$ be a Riemannian manifold of dimension $m+1$ and let $\Sigma$ be an oriented manifold of dimension $m\geq 4$ with a two-sided, stable, minimal immersion $f:\Sigma\rightarrow (M,g)$ such that $(\Sigma,g_{\Sigma})$ is complete. If $\mathcal{R}_g\geq \gamma\geq 0$, $\mathrm{Ric}_g\geq b>0$ and $|A|^2\leq\frac{\gamma p(m-p)+b}{\min\{p,m-p\}-1}$ with $p\in\{2,...,m-2\}$ then $$\mathcal{H}_2^p(\Sigma,g_{\Sigma})=\mathcal{H}_2^{m-p}(\Sigma,g_{\Sigma})=\{0\}.$$
\end{teo}

\begin{proof}
Let  $\alpha\subset \{1,...,m\}$ with $|\alpha|=p$. Since  $\Sigma$ is minimal, with the notation introduced in \eqref{pcc} we have $K_{\alpha}^2=K_{\star\alpha}^2$, hence $$K_{\alpha}^2= \leq \min\{p,m-p\}|A|^2\leq |A|^2+\gamma p(m-p)+b\leq \mathcal{B}_{\mathrm{res},p}+|A|^2+\mathrm{Ric}_g(\mathcal{N},\mathcal{N}).$$ 
Therefore $$\mathcal{B}_{\mathrm{res},p}+|A|^2+\mathrm{Ric}_g(\mathcal{N},\mathcal{N})+\min_{\alpha\subseteq \{1,...,m\},\ |\alpha|=k} -K^2_{\alpha}\geq 0$$ and the above inequality amounts to saying that $$\mathcal{B}_{p,\mathrm{res}}+\mathcal{B}_{\mathrm{ext},p}+\mathrm{Ric}_g(\mathcal{N},\mathcal{N})+|A|^2\geq 0.$$ We can thus conclude by the first point of Th. \ref{noL2forms} and the fourth point of Cor. \ref{immediately} that $\mathcal{H}^p_2(\Sigma,g_{\Sigma})=\mathcal{H}^{m-p}_2(\Sigma,g_{\Sigma})=\{0\}.$
\end{proof}

It is interesting to investigate the consequences of the previous result in the case of the round sphere $\mathbb S^{m+1}$. In this setting we have $\gamma=1$ and $b=m$. Given $2\leq p\leq \frac m2$,  the bound of the previous theorem becomes
$$
|A|^2\leq\beta(p,m):=\frac{p(m-p)+m}{p-1}.
$$
Note that $\beta$ is decreasing in $p$. Hence we have a stronger, but uniform in $p$, bound by taking $p=\left[\frac m2\right]$. After some standard algebraic manipulation we can see that
\begin{equation}\label{mari}
\beta\left(\left[\frac m2\right],m\right)>m\quad\text{if and only if}\ m\leq 7\ \text{or}\ m=9.
\end{equation}
Comparing with the non existence result of \cite[Corollary 1.3]{CMR} and the classification of \cite[Th. 1.1]{Mari}, we have that  Th. \ref{Abound} yields a new and complete vanishing in $\mathbb S^{m+1}$ for $m\in\{6,7,9\}$. More precisely:
\begin{cor}
Let $f:\Sigma\rightarrow \mathbb S^{m+1}$ be a complete, oriented, stable, minimally immersed hypersurface, with $m\in\{6,7,9\}$. If the bound
$$
|A|^2\leq\frac{\left[\frac m2\right](m-\left[\frac m2\right])+m}{\left[\frac m2\right]-1}=\left\{\begin{array}{lll}
 15/2&\text{if}&n=6,\\
19/2&\text{if}&n=7,\\
29/3 &\text{if}&n=9
\end{array}\right.,
$$
holds true, then for any $p\in\{0,...,m\}$ we have  $$\mathcal{H}_2^p(\Sigma,g_{\Sigma})=0.$$
\end{cor}

In addition to the above result we point out that for a fixed $p\in\{2,....,m/2\}$ the inequality $|A|\leq \beta(p,m)$ is generally weaker than the one required in \cite[Th. 1.1]{Mari}, namely $|A|\leq m$. Indeed for a fixed $p\in\{2,....,m/2\}$ we have $\beta(p,m)>m$ iff $m>p^2/2$. Therefore whenever the latter inequality holds true the vanishing of $\mathcal{H}^p_2(\Sigma,g_{\Sigma})$ is new and does not follows from the classification given in \cite{Mari}. For instance  we have  $\beta(2,m)>m$ for any $m> 2$ and $\beta(3,m)>m$ for any $m>4$.

\section{Examples}
In this section we apply our previous results to various explicit examples.

\subsection{Euclidean space} Theorem \ref{noL2forms} clearly applies when the ambient manifold is the Euclidean space. In view of both classical and recent developments on the Bernstein conjecture - some of them still under review at the time of writing - the result is trivial for $m=4$ and $m=5$, since in these cases it is proved that $\Sigma$ has to be a hyperplane, see \cite{CL1, CMR, CL2, CL3, Ma}.

\noindent On the other hand a celebrated result of Bombieri, De Giorgi and Giusti \cite{BDGG} showed that for $m\geq 8$ there exists minimal entire graphs (hence stable) in $\mathbb R^{m+1}$ that are not hyperplanes. Let us consider the case of $m=2n$ and let $\Sigma$ be the minimal graph described in \cite{BDGG}: it is a graph of the type $x_{m+1}=F(u,v)$, where $u=\sqrt{x_1^2+\cdots+x_n^2}$ and $v=\sqrt{x_{n+1}^2+\cdots+x_{m}^2}$. Because of the symmetries and minimality, $\Sigma$ has two principal curvatures: $\pm k$, both with multiplicity $n$. By Proposition \ref{zeropp}, point 1 and Theorem \ref{noL2forms}, point 1 we have that for any $p\leq\sqrt m$ 

$$\mathcal{H}^p_2(\Sigma,g_{\Sigma})=\mathcal{H}^{m-p}_2(\Sigma,g_{\Sigma})=\{0\}.$$

\subsection{Spaces with nonnegative curvature operator}
Famous examples of Riemannian manifolds with nonnegative curvature operator are $\mathbb R^m$, $\mathbb S^m$ and $\mathbb{C}\mathbb{P}^m$, each one endowed with its standard metric. In \cite[Section 6.5]{BoKa} it is showed that $\mathbb S^{2n+1}$ with a Berger metric with parameter $\vep$ has non-negative curvature operator if and only if $0<\vep<\frac{2n+2}{2n+1}$ (see Lemma \ref{sph-berger} below). Moreover the Riemannian product of any two (or more) Riemannian manifolds with non-negative curvature operator produces a new Riemannian manifold with non-negative curvature operator.  We are interested in considering ambient manifolds with non-negative curvature operator and of dimension $5$ because in this case the condition $|A|^2-K_{\alpha}^2\geq 0$ is automatically satisfied, see \cite[Lemma1]{Tanno}.   
However  if the dimension is less or equal than $6$ and the ambient manifold has non-negative sectional curvature and uniformly positive Ricci curvature, then there is no complete, orientable, immersed, stable minimal hypersurface, see \cite[Corollary 1.3]{CMR}. 

Examples of dimension $5$ of particular interest are: $\mathbb S^2\times\mathbb R^3$, $\mathbb S^3\times\mathbb R^2$, $\mathbb S^4\times\mathbb R$, $\mathbb S^5$, $\mathbb S^2\times\mathbb S^3$, $\mathbb S^2\times\mathbb S^2\times\mathbb R$, $\mathbb {CP}^2\times\mathbb R$. 

Examples of dimension $6$ of particular interest are: $\mathbb S^2\times\mathbb R^4$, $\mathbb S^3\times\mathbb R^3$, $\mathbb S^4\times\mathbb R^2$, $\mathbb S^5\times\mathbb R$, $\mathbb S^6$, $\mathbb S^2\times\mathbb S^4$, $\mathbb S^3\times\mathbb S^3$, $\mathbb S^2\times\mathbb S^2\times\mathbb S^2$, $\mathbb S^2\times\mathbb S^3\times\mathbb R$, $\mathbb S^2\times\mathbb S^2\times\mathbb R^2$ and $\mathbb {CP}^2\times\mathbb R^2$.

Note that in the above list, when more spherical factors appears, they can have different radii. Moreover the factors $\mathbb S^3$ (resp. $\mathbb S^5$) could be endowed with a Berger metric $g_{\vep}$ for any $\vep<\frac 43$ (resp. $\vep<\frac 65$). 

Because of the non-negativity of the curvature operator, Theorems \ref{noL2forms} applies to all such ambient manifolds. However, by \cite[Theorem 1.2]{CMR}, the problem of looking for an orientable stable minimal hypersurface becomes trivial in some of them. Let us give a detailed picture:

\begin{enumerate}
\item  $\mathbb S^5$, $\mathbb S^2\times\mathbb S^3$, $\mathbb S^6$, $\mathbb S^2\times\mathbb S^4$, $\mathbb S^3\times\mathbb S^3$ and $\mathbb S^2\times\mathbb S^2\times\mathbb S^2$ have non-negative sectional curvature, but uniformly positive Ricci curvature, therefore there is no complete, orientable, immersed, stable minimal hypersurface, by \cite[Corollary 1.3]{CMR}.

\item $\mathbb S^4\times\mathbb R$, $\mathbb {CP}^2\times\mathbb R$, $\mathbb S^2\times\mathbb S^2\times\mathbb R$, $\mathbb S^5\times\mathbb R$ and $\mathbb S^2\times\mathbb S^3\times \mathbb R$ have non-negative sectional curvature, but uniformly positive bi-Ricci curvature.  Therefore every complete, orientable, immersed, minimal stable hypersurface in any of these spaces is an horizontal slice as we can prove with more generality in the following result.

\begin{prop}
Let $(N,g_N)$ be a compact Riemannian manifolds with dimension $n\leq 5$, nonnegative sectional curvature and uniformly positive Ricci and bi-Ricci curvature. Let $M=N\times\mathbb R$ endowed with the product metric. Then every complete, two-sided, immersed, stable minimal hypersurface of $M$ is a horizontal slice $N\times\{t_0\}$ for some $t_0\in\mathbb R$.
\end{prop}
\begin{proof}
By the hypothesis on $(N,g_N)$ we have that $M$ has uniformly positive bi-Ricci curvature and $\mathrm{Ric}_M\geq 0$ with $\mathrm{Ric}_M(X,X)=0$ if and only if $X$ is tangent to the factor $\mathbb R$. Applying \cite[Theorem 1.2]{CMR} we have that every complete, orientable, immersed, stable minimal hypersurface of $M$ is compact. Compactness together with stability condition and $\mathrm{Ric}_M\geq 0$ implies that every such hypersurface satisfies 
$$
|A|^2=\mathrm{Ric}_M(\mathcal N,\mathcal N)= 0.
$$
In particular we deduce that $\mathcal N$ is everywhere parallel to the factor $\mathbb R$. Therefore the result follows.
\end{proof}

\item $\mathbb S^3\times\mathbb R^2$, $\mathbb S^2\times\mathbb R^3$, $\mathbb S^4\times\mathbb R^2$, $\mathbb S^3\times\mathbb R^3$, $\mathbb S^2\times\mathbb R^4$, $\mathbb S^2\times\mathbb S^2\times \mathbb R^2$ and $\mathbb{CP}^2\times\mathbb R^2$ do not satisfy the hypothesis of \cite[Theorem 1.2]{CMR}. To the best of our knowledge, Theorem \ref{noL2forms} provides the most accurate description of the topology of oriented stable minimal hypersurfaces in these ambient manifolds. 
\end{enumerate}

\subsection{Berger spheres}

We conclude this section by describing a well known class of ambient manifolds. Let $\vep>0$ and let $g_{\vep}$ be the Berger metric of parameter $\vep$ on $\mathbb S^{2n+1}$ with $n\geq 2$. This metric can be described as follows: embedd canonically $\mathbb S^{2n+1}\subset\mathbb R^{2n+2}\equiv\mathbb C^{n+1}$, let $J$ be the complex structure of $\mathbb C^{n+1}$ and, for any $p\in\mathbb S^{2n+1}$, let $\xi_p=Jp$. Then $\xi$ is a vector field tangent to the sphere called \emph{Reeb vector field}. Let $\eta$ be the $1$-form dual to $\xi$, and $\sigma=g_1$ be the standard round metric on the sphere then we define
$$
g_{\vep}=\sigma+(\vep-1)\eta^2.
$$
By definition $g_{\vep}(\xi,\xi)=\vep$. Let $\mathcal H$ be the $2n$-dimensional distribution orthogonal to $\xi$, then, with an abuse of notation, we call $J$ the complex structure on $\mathcal H$. Let $\hat\xi=\vep^{-\frac 12}\xi$. 

\begin{lemma}\label{sph-berger}
Let $\vep>0$, consider the Berger sphere $(\mathbb S^{2n+1},g_{\vep})$, then for any $X,Y$ orthogonal unit vector in $\mathcal H$ with $Y\neq\pm JX$ the following holds:
\begin{enumerate}
\item the sectional curvatures are:
$$
\mathrm{sec}_{g_{\delta}}(\hat\xi,X)=\vep,\quad \mathrm{sec}_{g_{\delta}}(X,JX)=4-3\vep,\quad \mathrm{sec}_{g_{\delta}}(X,Y)= 1,
$$
in particular $\mathrm{sec}_{g_{\delta}}>0$ for any $\vep<\frac 43$;
\item the Ricci curvature tensor is:
$$
\mathrm{Ric}_{g_{\delta}}(\hat\xi,\hat\xi)=2n\vep,\quad \mathrm{Ric}_{g_{\delta}}(X,X)=2n+2-2\vep ,
$$
and zero elsewhere, in particular $\mathrm{Ric}_{g_{\delta}}>0$ for any $\vep <n+1$;
\item the scalar curvature is:
$$
s_{g_\vep}=2n(2n+2-\vep).
$$
\item the eigenvalues of the curvature operator $\mathcal R_{g_{\vep}}$ are $1$ with multiplicity $n(2n+1)$ if $\vep=1$, otherwise they are

\begin{itemize}
\item $\vep$ with multiplicity $n(n+1)$;
\item $2-\vep$ with multiplicity $n^2-1$;
\item $2n+2-(2n+1)\vep$ with multiplicity $1$,
\end{itemize}
In particular the curvature operator is positive (resp. non-negative) defined for any $\vep<\frac{2n+2}{2n+1}$ (resp. $\vep\leq \frac{2n+2}{2n+1}$).\end{enumerate}
\end{lemma}
\begin{proof} The case $\vep=1$ is trivial, so from now on let $\vep\neq 1$. It is well known that the metric $\tilde g=|(1-\vep^2)|g_{\vep}$ on $\mathbb S^{2n+1}$ is the induced metric of a geodesic sphere in a complex space form: in particular if $0<\vep<1$ (resp. $\vep>1$), then $\tilde g$ is the induced metric of a geodesic sphere of $\mathbb{CP}^{n+1}$ (resp. $\mathbb{CH}^{n+1}$) of radius $\rho$ such that $\cos^2(\rho)=\vep$ (resp. $\cosh^2(\rho)=\vep$). The sectional curvatures of the metrics $\tilde g$ can be found in \cite[Section 6.4]{BoKa}. Rescaling by a factor $|1-\vep^2|$ gives the proof of item $1$. Items $2.$ and $3.$ follows directly. Finally, the eigenvalues of the curvature operator of $\tilde g$ are described in \cite[Section 6.5]{BoKa} (please be aware of different notations: in \cite{BoKa} the authors considered metrics on a $(2n-1)$-dimensional sphere and their curvature operator is twice ours). Once again, the result follows after a rescaling by a factor $|1-\vep^2|$.
\end{proof}

Thanks to Lemma \ref{sph-berger}, we are able to apply the results of Section \ref{sez2} to stable minimal hypersurfaces of the Berger spheres.

\begin{teo} \label{min-berger}
Let $f:\Sigma\rightarrow (\mathbb S^{2n+1},g_{\vep})$ be a complete oriented stable minimally immersed hypersurface of the Berger sphere $(\mathbb S^{2n+1},g_{\vep})$, with $n\geq 2$. 
\begin{enumerate}
\item If $0<\vep\leq 2n+2$, then $\Sigma$ has no non-trivial $L^2$-harmonic spinors.
\item If $0<\vep\leq\max\left\{\frac{2n+2}{2n+1},\ \frac{4(2n^2+n+6)}{8n^2+n+18}\right\}$ and there exist a $p\in\{2,\dots,n\}$ such that  $|A|^2-K^2_{\alpha}\geq 0$ for any $\alpha\subset \{1,...,m\}$ with $|\alpha|=p$, then $$\mathcal{H}^p_2(\Sigma,g_{\Sigma})=\mathcal{H}^{2n-p}_2(\Sigma,g_{\Sigma})=\{0\}.$$
\item If $\frac{8(n-1)}{17m-14}\leq \vep\leq\frac{4(11n-8)}{35n-26}$ and $|A|^2-K^2_{\alpha}\geq 0$ for any $\alpha\subset \{1,...,m\}$ with $|\alpha|=2$, then $$\mathcal{H}^2_2(\Sigma,g_{\Sigma})=\mathcal{H}^{2n-2}_2(\Sigma,g_{\Sigma})=\{0\}.$$
\end{enumerate}
\end{teo}
\begin{proof}
\begin{enumerate}
\item The choice of the parameter $\vep$ guarantees that the scalar curvature of the corresponding Berger sphere is non-negative by Lemma \ref{sph-berger}. Let now $s$ be a $L^2$-harmonic spinor on $\Sigma$. By Theorem \ref{spin} we know that $s$ has constant length. If $s$ is non-trival, by Corollary \ref{s-flat} we have that $\Sigma$ is totally geodesic which leads to a contradiction. In fact, if $\vep=1$  it is well know that in the round sphere the totally geodesic hypersurfaces are not stable. On the other hand, when $\vep\neq 1$, the corresponding Berger sphere does not have totally geodesic hypersurfaces by \cite[Theorem A]{OV}.

\item If $\vep\leq\frac{2n+2}{2n+1}$, thanks to Lemma \ref{sph-berger} we have that $\mathcal R_{g_{\vep}}\geq 0$ and $\mathrm{Ric}_{g_{\vep}}>0$, hence the result follows by Theorem \ref{noL2forms}, point 1. Note that $\frac{2n+2}{2n+1}\leq \frac{4(2n^2+n+6)}{8n^2+n+18}$ if and only if $n\leq 6$. So from now on let $n>6$ and $\frac{4(2n^2+n+6)}{8n^2+n+18}<\vep<\frac{2n+2}{2n+1}$. It is immediate to check that for $p\in\{2,\dots,n\}$ we have 
$$\varepsilon_{2n,p}\geq\varepsilon_{2n,n}=\frac{2n^2+n+6}{2n(n-1)},$$
where the constant $\varepsilon_{2n,p}$ was defined  \eqref{nonsharp}. Since in particular we are assuming $1<\vep<\frac 43$, by Lemma \ref{sph-berger} we can apply  Theorem \ref{noL2forms}, point 2 with $a=4-3\vep$ and $b=\vep$.

\item Let $c_{2n}=\frac{11n-8}{2(n-1)}$ be the constant defined in \eqref{sharp}. Using Lemma \ref{sph-berger} and some standard algebraic manipulations we can see that with this choice of $\vep$ the curvature hypothesis $0<a\leq b\leq c_{2n}a$ is satisfied. Therefore we can apply Theorem \ref{noL2forms}, point 3.
\end{enumerate}
\end{proof}
\begin{rem}\label{rmkcond}
The curvature conditions used in Theorem \ref{noL2forms} - namely $\mathcal R_g\geq 0$, $0<a\leq b\leq \varepsilon_{2n,n}a$ and $0<a\leq b\leq c_{2n}a$ - have clearly a big intersection, but the discussion about the parameter $\vep$ in the proof of Theorem \ref{min-berger} shows that, in general, they may not be included one into the others. See the Figure \ref{fig01}  for a schematic summary in the case of Berger spheres.
\end{rem}

\begin{figure}[h]\
\centering
\includegraphics[width=0.9\textwidth]{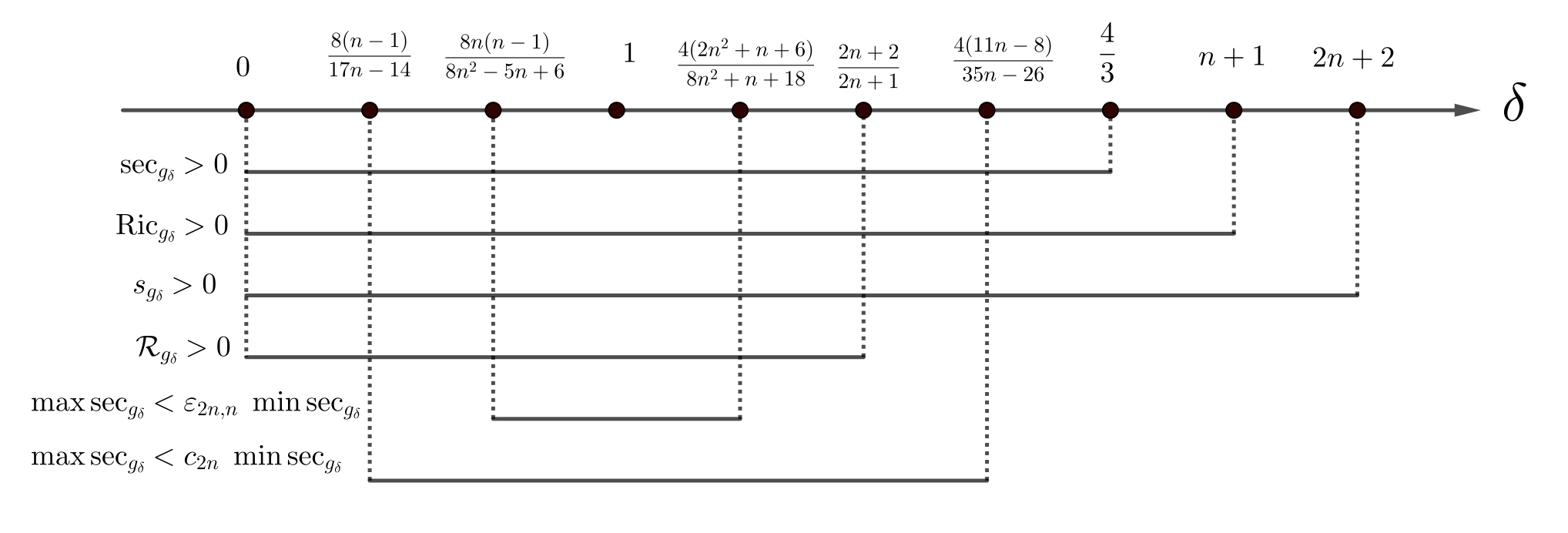}
\caption{Comparison of different notion of curvature for the Berger spheres $(\mathbb S^{2n+1},g_{\vep})$ for $n>6$. When $n<6$ (resp. $n=6$) the order of the special values of $\vep$ is the same except that $\frac{2n+2}{2n+1}$ and $\frac{4(2n^2+n+6)}{8n^2+n+18}$ are interchanged (resp. coincide) .}
\label{fig01}
\end{figure}

\section{Strongly stable constant mean curvature hypersurfaces}
Many of the ideas developed in previous sections can be easily adapted to the study of stable constant mean curvature (CMC for short) hypersurfaces. We recall that a two-sided, immersed hypersurface $f:N\rightarrow (M,g)$ is said strongly stable CMC if its mean curvature $H$ is a non-zero constant and if for any $u$ smooth function on $\Sigma$ with compact support we have 
\begin{equation}\label{stab-op}
\int_{N}(|A|^2+\mathrm{Ric}_g(\mathcal N,\mathcal N)u^2\dvol_{g_{N}}\leq\int_{N}|\nabla u|^2\dvol_{g_{N}}
\end{equation}
We point out that we are not requiring that $\int_{N}u\ dvol_{g_{N}}=0$ as in the usual definition of stability (sometimes called weak stability) for CMC hypersurfaces, see \cite{ENR} and the reference therein. Since condition \eqref{stab-op} is the same condition defining a two-sided, stable minimal hypersurface, results of previous sections can be extended with minor modification to the case of a strongly stable CMC hypersurface.  In this Section we are using the following notation: $H=\sum_i\lambda_i$, where $\lambda_i$ are the principal curvatures of the hypersurface.

The next result extends Theorem \ref{spin} to the case of CMC hypersurfaces and shows that, in this case, the ambient manifold may have negative scalar curvature, but controlled by $H$.

\begin{teo}
Let $f:N\rightarrow (M,g)$ be a spinnable, two-sided, complete, strongly stable CMC immersed hypersurface with mean curvature $H>0$. Let $P_{\mathrm{Spin(n)}}(N)\rightarrow N$ be an arbitrarily fixed Riemannan spin structure on $N$ and let $(\Sigma N,\tau)\rightarrow N$ and $\eth$ be the corresponding spinor bundle and spin-Dirac operator. If the scalar curvature of the ambient manifold satisfies $s_g+H^2\geq 0$, then every $L^2$-harmonic spinor of $(\Sigma N, \tau)\rightarrow N$ has constant length and consequently 
$$\dim\left(\ker(\eth)\cap L^2(N,\Sigma N,g_N,\tau)\right)\leq 2^{n/2}$$ if $n$ is even whereas
$$\dim\left(\ker(\eth)\cap L^2(N,\Sigma N,g_N,\tau)\right)\leq 2^{(n-1)/2}$$ if $n$ is odd. If $(\Sigma,g_{\Sigma})$ has infinite volume, then $(N,g_N)$ carries no non-trivial $L^2$-harmonic spinors.

\end{teo}

\begin{proof}
We can repeat the proofs of Theorem \ref{spin} and Corollary \ref{spin-cor} with the following modifications: when $H\neq 0$, \eqref{GC} becomes  
$$
\mathrm{Ric}_g(\mathcal N,\mathcal N)+|A|^2=\frac 12(-s_{g_N}+s_g+H^2+A^2),
$$
hence \eqref{wer} becomes 
$$
\frac{1}{4}s_{g_N}+\frac{1}{2}(\mathrm{Ric}_g(\mathcal{N},\mathcal{N})+|A|^2)=\frac{1}{4}(s_g+H^2+|A|^2)\geq 0 .
$$
\end{proof}
Finally let us study $L^2$-harmonic forms on strongly stable CMC hypersurfaces.

\begin{teo}
\label{noL2forms-CMC}
Let $(M,g)$ be a Riemannian manifold of dimension $m+1$ and let $\Sigma$ be an oriented manifold of dimension $m\geq 4$ with a two-sided, strongly stable, CMC immersion $f:\Sigma\rightarrow (M,g)$ such that $(\Sigma,g_{\Sigma})$ is complete. { Let $H$ be its mean curvature and let  $2\leq p\leq \frac m2$}.
\begin{enumerate}
\item If $\mathcal{R}_g\geq0$ (or more generally there exists $\gamma\in \mathbb{R}$ such that $\mathcal{R}_g\geq \gamma$ and $p(m-p)\gamma+\mathrm{Ric}_g(\mathcal{N},\mathcal{N})\geq 0$) and  $|A|^2+K_{\alpha}(H-K_{\alpha})\geq 0$ for any $\alpha\subset \{1,...,m\}$ with $|\alpha|=p$, then every $L^2$-harmonic $p$-form on $(\Sigma,g_\Sigma)$ has constant length.  If in addition $\Sigma$ is not totally geodesic or $\mathrm{Ric}_g(\mathcal{N},\mathcal{N})$ is somewhere positive on $f(\Sigma)$, then $$\mathcal{H}^p_2(\Sigma,g_{\Sigma})=\mathcal{H}^{m-p}_2(\Sigma,g_{\Sigma})=\{0\}.$$
\item If $\mathrm{sec}_g\in [a,b]$ with $0<a\leq b\leq \varepsilon_{m,p}a$ with $\varepsilon_{m,p}$ the constant defined in \eqref{nonsharp} and $|A|^2+K_{\alpha}(H-K_{\alpha})\geq 0$ for any $\alpha\subset \{1,...,m\}$ with $|\alpha|=p$, then $$\mathcal{H}^p_2(\Sigma,g_{\Sigma})=\mathcal{H}^{m-p}_2(\Sigma,g_{\Sigma})=\{0\}.$$

\item If $\mathrm{sec}_g\in [a,b]$ with $0<a\leq b\leq c_ma$, $c_m$ the constant defined in \eqref{sharp} and $|A|^2+K_{\alpha}(H-K_{\alpha})\geq 0$ for any $\alpha\subset \{1,...,m\}$ with $|\alpha|=2$, then $$\mathcal{H}^2_2(\Sigma,g_{\Sigma})=\mathcal{H}^{m-2}_2(\Sigma,g_{\Sigma})=\{0\}.$$
\end{enumerate}
\end{teo}

\begin{proof}
The proof is analogous to that of Theorem \ref{noL2forms} except for the fact that, since $H\neq 0$, now $K_{\star\alpha}=H-K_{\alpha}$ holds true for any $\alpha\subset\{1,\dots,m\}$.
\end{proof}

We conclude with the following comment.

\begin{rem}
Concerning the condition $|A|^2+K_{\alpha}(H-K_{\alpha})\geq 0$ we can make similar comments to those of Remark \ref{L21}. Moreover, since now $H\neq 0$, we have that the above inequality is automatically satisfied without trivializing the problem also when the hypersurface is convex.
\end{rem}

\vspace{1 cm}
 
\textbf{Acknowledgements.} The authors are grateful to Alessandro Savo and Barbara Nelli for interesting discussions. The authors are partially supported by INDAM-GNSAGA ''Gruppo Nazionale per le Strutture Algebriche, Geometriche e le loro Applicazioni'' of the Istituto Nazionale di Alta Matematica ''Francesco Severi''. The second author is partially supported by PRIN project no. 20225J97H5.

\vspace{1 cm}

Dipartimento di Matematica, Sapienza Universit\`a di Roma, bei@mat.uniroma1.it\\

Dipartimento di Ingegneria e Scienze dell'informazione e Matematica, Universit\`a degli studi dell'Aquila, giuseppe.pipoli@univaq.it

\end{document}